\documentclass[12pt,draftcls,onecolumn]{IEEEtran}

\IEEEoverridecommandlockouts                              

\makeatletter
\let\NAT@parse\undefined
\makeatother

\usepackage{graphicx}
\usepackage{bm}
\usepackage{epsfig}
\usepackage{amsthm}
\usepackage{amssymb,amsfonts,amsmath}
\usepackage{subfigure}
\usepackage{setspace}
\usepackage{enumerate}
\usepackage[numbers]{natbib}
\usepackage{color}
\usepackage{dsfont}
\usepackage{verbatim}

\newtheorem{theorem}{Theorem}
\newtheorem{lemma}{Lemma}
\newtheorem{remark}{Remark}

\newtheorem{example}{Example}

\newtheorem{corollary}{Corollary}

\newcommand{\cred}[1]{{\color{red}{#1}}}

\def \s {\sigma}

\def \< {\langle}
\def \> {\rangle}
\def \d {\delta}
\def \O {\Omega}

\def \S{\mathcal{S}}
\def \F{\mathcal{F}}
\def \B{\mathcal{B}}
\def \~{\tilde}

\def \Lie{{\rm Lie}}
\title{\LARGE \bf
Analyzing Controllability of Bilinear Systems on Symmetric Groups:
Mapping Lie Brackets to Permutations
}

\author{Wei Zhang, ~\IEEEmembership{Student Member,~IEEE,} and Jr-Shin~Li,~\IEEEmembership{Senior Member,~IEEE,}
\thanks{*This work was supported in part by the National Science Foundation under the awards CMMI-1462796 and ECCS-1509342, and by the Air Force Office of Scientific Research under the award FA9550-17-1-0166.}
\thanks{W. Zhang is with the Department of Electrical and Systems Engineering, Washington University,
		St. Louis, MO 63130, USA
        {\tt\small wei.zhang@wustl.edu}}%
\thanks{J.-S. Li is with the Department of Electrical and Systems Engineering, Washington University,
        St. Louis, MO 63130, USA
        {\tt\small jsli@wustl.edu}}%
}

\begin{document}

\maketitle

\begin{abstract}
	Bilinear systems emerge in a wide variety of fields as natural models for dynamical systems ranging from robotics to quantum dots. Analyzing controllability of such systems is of fundamental and practical importance, for example, for the design of optimal control laws, stabilization of unstable systems, and minimal realization of input-output relations. Tools from Lie theory have been adopted to establish controllability conditions for bilinear systems, and the most notable development was the Lie algebra rank condition (LARC). However, the application of the LARC may be computationally expensive for high-dimensional systems. In this paper, we present an alternative and effective algebraic approach to investigate controllability of bilinear systems. The central idea is to map Lie bracket operations of the vector fields governing the system dynamics to permutation multiplications on a symmetric group, so that controllability and controllable submanifolds can be characterized by permutation cycles. The method is further applicable to characterize controllability of systems defined on undirected graphs, such as multi-agent systems with controlled couplings between agents and Markov chains with tunable transition rates between states, which in turn reveals a graph representation of controllability through the graph connectivity.
\end{abstract}

\section{Introduction}
The impact of differential geometry 
in the context of nonlinear control became salient in the early 1970s \cite{Brockett14}. It was driven by the need of extending linear control theory into a nonlinear setting, and Lie theory, in particular, provided a suitable and powerful toolkit on this account. 
Brockett, Jurdjevic, Sussmann and other pioneers introduced the theory of Lie groups and their associated Lie algebras to the domain of nonlinear control to interpret 
notions, such as controllability, reachability, observability, and realization, for 
nonlinear control systems \cite{Brockett72, Jurdjevic72, Hirschorn75, Brockett76, Hermann77, Baillieul81, Jakubczyk90, Manikonda00}. Their seminal works have led to the development of necessary and/or sufficient conditions for the characterization of these fundamental properties, and inspired many systematic and influential approaches to the design of control laws that steer and stabilize nonlinear control systems \cite{Isidori81, Krener83, Sastry89, Isidor90, Bloch92, Sontag1995, Persis01} and to the understanding of sufficient statistics in nonlinear filtering and various aspects in quantum control \cite{Huang83,Khaneja02,Li_PRA06}. 

One prominent application of differential geometric techniques in control theory has been to the 
controllability analysis of nonlinear systems, and the most notable development is, unarguably, the Lie algebra rank condition (LARC) \cite{Brockett72, Jurdjevic96}. Examining controllability using the LARC requires the computation of iterative Lie brackets of the drift and control vector fields and then the verification of linear independence among all of the resulting vector fields. This process is computationally expensive for high-dimensional systems. In this paper, we present a new, alternative framework to investigate controllability of right-invariant bilinear systems. In particular, we study the control systems 
governed by the vector fields that form a magma structure $\mathfrak{h}=(M,[\cdot\ ,\ \cdot])$ under the Lie bracket operation, namely, $[f,g]\in M$ for any $f,g\in M$. Note $M$ here is the set consisting of the basis elements of Lie$(M)$, the Lie algebra generated by $M$. This structure represents broad classes of control systems, such as systems evolving on compact Lie groups \cite{Jurdjevic96,Sachkov2009} and multi-agent systems governed by reciprocal interaction laws \cite{Chen2015}. We develop an effective algebraic approach to analyze controllability of such bilinear systems. The central idea is to map Lie bracket operations of the vector fields governing the system dynamics to permutation multiplications on a symmetric group, so that controllability and controllable submanifolds can be characterized by permutation cycles. The method can be further adopted to characterize controllability of systems defined on undirected graphs, such as multi-agent systems with controlled couplings between agents \cite{Chen14}, Markov chains with tunable transition rates between states \cite{Brockett08}, and quantum networks \cite{Khaneja02}. Moreover, this new framework reveals a graph representation of controllability through the graph connectivity. 

The paper is organized as follows. In Section \ref{sec:SO(n)_controllability}, we introduce the proposed algebraic approach via the study of controllability for the system defined on the Lie group SO($n$), from which we illuminate the idea of mapping Lie brackets on SO($n$) to permutations on $S_n$, the symmetric group of order $n$. In particular, we construct an algebraic necessary and sufficient controllability condition for the system on SO($n$) in terms of the length of permutation cycles. In Section \ref{sec:*}, we further define a monoid structure on $S_n$, which induces an equivalence relation that is used to explicitly characterize controllability based on the entire set of control vector fields and the controllable submanifold of the system on SO$(n)$. 
Finally, in Section \ref{sec:graph}, we extend the scope of this novel framework to study the systems defined on undirected graphs, including multi-agent systems and stochastic systems described 
by Markov chains. 

\section{Interpreting Controllability of Systems on SO$(n)$ over Symmetric Groups}
\label{sec:SO(n)_controllability}
In this section, we introduce a new algebraic framework for analyzing controllability of bilinear systems governed by the vector fields that form a magma structure. We begin with presenting our method through the study of controllability for the system defined on SO$(n)$, through which the idea of mapping Lie brackets into permutation multiplications is enlightened. We briefly review the classical controllability results characterized by the LARC for control systems on compact, connected Lie groups. Motivating examples are provided to illustrate the inefficiency and complexity of using the LARC for examining controllability. We then review some essential tools of the symmetric group theory and construct a necessary and sufficient controllability condition for systems defined on SO$(n)$. In particular, we establish a correspondence between the Lie bracket operations on the Lie algebra $\mathfrak{so}(n)$ and permutation multiplications on the symmetric group $S_n$. This gives rise to an explicit condition for effective examination of controllability in terms of the length of permutation cycles.


\subsection{Controllability of Systems on Compact Lie groups}
Controllability of a system evolving on a compact, connected Lie group has been extensively studied \cite{Brockett72, Jurdjevic72, Jurdjevic96, Sastry99}. The central idea lies in the investigation of the equivalence between the Lie algebra generated by the control (and the drift) vector fields and the underlying Lie algebra associated with the Lie group. Let's consider the time-invariant bilinear control system defined on a compact, connected Lie group $G$ of the form,
\begin{align}
	\label{eq:general_form}
	\dot{X}(t)=AX(t)+\left[\sum_{i=1}^mu_i(t)B_i\right]X(t),\quad X(0)=I,
\end{align}
where $X(t)\in G$ denotes the state, $A,B_1,\dots,B_m$ are 
elements in the Lie algebra $\mathfrak{g}$ of $G$, $I$ is the identity element of the Lie group $G$, and $u_i(t)\in\mathbb{R}$ are piecewise constant control functions for $i=1,\ldots,m$. 
We denote the Lie algebra generated by the set $\{A,B_1,\dots,B_m\}$ as Lie$\{A,B_1,\dots,B_m\}$. This is the smallest linear subspace of $\mathfrak{g}$ containing $\{A,B_1,\dots,B_m\}$, which is closed under the Lie bracket operation, i.e., $[C,D]=CD-DC$ for all $C,D\in\mathfrak{g}$.

\begin{theorem}
	\label{thm:lie_rank}
	The system in \eqref{eq:general_form} is controllable on the 
Lie group $G$ if and only if $\Lie(\mathcal{F})=\mathfrak{g}$, where $\mathcal{F}=\{A,B_1,\dots,B_m\}$. 
\end{theorem}
\begin{proof}
	See \cite{Brockett72, Khaneja00}.
\end{proof}

\subsubsection{Basics of the Lie Algebra $\mathfrak{so}(n)$}
\label{sec:basics}
Let $E_{ij}\in\mathbb{R}^{n\times n}$ denote the matrix whose $ij^{\rm th}$ entry is $1$ and the others are $0$, and let $\O_{ij}=E_{ij}-E_{ji}$, then 
\begin{align*}
\O_{ij}=
\begin{cases}
-\O_{ji},\ \text{if\ } i\neq j,\\
0, \qquad \text{if\ } i=j,
\end{cases}
\end{align*}
for all $i,j=1,\dots,n$. The set $\mathcal{B}=\{\O_{ij}:1\leq i<j\leq n\}$ forms a standard basis of $\mathfrak{so}(n)$, which has the dimension $n(n-1)/2$. For simplicity, we will adopt the following notations throughout this paper: 
\begin{enumerate}
	\item[$\bullet$] $\B$: the standard basis of $\mathfrak{so}(n)$; 
	\item[$\bullet$] $\F$: the set of the control vector fields of a given system on SO$(n)$ and $\F\subseteq\B$.
\end{enumerate}

We first observe the following Lie bracket relations of the basis elements in $\mathfrak{so}(n)$.

\begin{lemma}
	\label{lem:son}
	The Lie bracket of $\O_{ij}$ and $\O_{kl}$ 
	satisfies the relation $[\O_{ij},\O_{kl}]=\d_{jk}\O_{il}+\d_{il}\O_{jk}+\d_{jl}\O_{ki}+\d_{ik}\O_{lj}$, where $\d$ is the Kronecker delta function, i.e.,
	\begin{align*}
		\d_{mn}=\begin{cases} 1 \quad {\rm if\ } m=n, \\ 0 \quad {\rm if\ } m\neq n. \end{cases}
	\end{align*}
\end{lemma}
\begin{proof}
	Notice that $E_{ij}E_{kl}=\d_{jk}E_{il}$, and hence $[E_{ij},E_{kl}]=\delta_{jk}E_{il}-\delta_{li}E_{kj}$. Following the bilinearity of the Lie bracket, we get
\begin{align*}
	[\O_{ij},\O_{kl}]&=[E_{ij}-E_{ji},E_{kl}-E_{lk}]\\
	&= [E_{ij},E_{kl}]-[E_{ij},E_{lk}]-[E_{ji},E_{kl}]+[E_{ji},E_{lk}] \\
	&= \delta_{jk}E_{il}-\delta_{li}E_{kj}-\delta_{jl}E_{ik}+\delta_{ki}E_{lj}\\
	&\quad -\delta_{ik}E_{jl}+\delta_{lj}E_{ki}+\delta_{il}E_{jk}-\delta_{kj}E_{li}\\
	&= \delta_{jk}\O_{il}+\delta_{il}\O_{jk}+\delta_{jl}\O_{ki}+\delta_{ik}\O_{lj}. 
\end{align*}
It follows that for any $\O_{ij},\O_{kl}\in\mathcal{B}$, the Lie bracket $[\O_{ij},\O_{kl}]\neq 0$ if and only if $i=k$, $i=l$, $j=k$, or $j=l$.
\end{proof}
\noindent

Because SO$(n)$ is compact and connected, Theorem \ref{thm:lie_rank} can be applied to check controllability of a system defined on SO($n$). However, the examination of the LARC requires repeated Lie bracket operations. 
The inefficiency and complexity of this procedure is illustrated by the following examples.

\subsubsection{Complexity of the Application of LARC} 
\label{sec:LARC_examples}

\begin{example} \rm
\label{ex:so(5)}
Consider the system evolving on SO$(5)$, given by
\begin{align}
	\label{eq:ex_so(5)}
	\dot{X}(t)=\sum_{i=1}^4u_i(t)\O_{i,i+1}X(t),\quad X(0)=I,
\end{align}
where $\mathcal{F}=\{\O_{i,i+1}: i=1,\ldots,4\}$ is the set of control vector fields evaluated at the identity matrix $I$. Then, Lie$(\mathcal{F})$ is a Lie subalgebra of $\mathfrak{so}(5)$. Because $\mathfrak{so}(5)$ is a ten-dimensional real vector space and $\mathcal{F}$ contains four linearly independent elements of $\mathcal{B}$, the system in \eqref{eq:ex_so(5)} is controllable if the rest of the six basis elements of $\mathfrak{so}(5)$ can be generated by iterated Lie brackets of elements in $\F$. By applying the Lie bracket properties in Lemma \ref{lem:son} repeatedly, we obtain
\begin{align*}
&[\O_{12},\O_{23}]=\O_{13},\ [\O_{23},\O_{34}]=\O_{24},\ [\O_{34},\O_{45}]=\O_{35},\\
&[\O_{12},\O_{24}]=[\O_{12},[\O_{23},\O_{34}]]=\O_{14},\\
&[\O_{23},\O_{35}]=[\O_{23},[\O_{34},\O_{45}]]=\O_{25},\\
&[\O_{12},\O_{25}]=[\O_{12},[\O_{13},[\O_{34},\O_{45}]]]=\O_{15}.
\end{align*}
Because $\{\O_{13},\O_{14},\O_{15},\O_{24},\O_{25},\O_{35}\}\cup\F=\B$, the system in \eqref{eq:ex_so(5)} is controllable on SO$(5)$ by Theorem \ref{thm:lie_rank}.

For this low-dimensional system, it requires computations up to triple brackets in order to analyze controllability. In general, it may need a large number of Lie bracketing iterations in order to conclude controllability by using the LARC for systems defined on higher dimensional special orthogonal groups.
\end{example}

\begin{example} 
\label{ex:so(5)_2}
\rm
Consider the system evolving on SO$(5)$ driven by three controls, 
given by
\begin{align}
	\label{eq:ex_so(5)_2}
	\dot{X}(t) &= \left[u_1(t)\O_{12}+u_2(t)\O_{23}+u_3(t)\O_{45}\right]X(t), \nonumber\\
	X(0) &= I.
\end{align}
In this case, we have $\mathcal{F}=\{\O_{12},\O_{23},\O_{45}\}$. Then, the single, 
\begin{align*}
	&[\O_{12},\O_{23}]=\O_{13}, \\
	&[\O_{12},\O_{45}]=[\O_{23},\O_{45}]=0,
\end{align*}
and double Lie brackets,
\begin{align*}
	&[\O_{13},\O_{12}]=[[\O_{12},\O_{23}],\O_{12}]=\O_{23},\\
	&[\O_{23},\O_{13}]=[\O_{23},[\O_{12},\O_{23}]]=\O_{12},\\
	&[\O_{13},\O_{45}]=[[\O_{12},\O_{23}],\O_{45}]=0,
\end{align*}
result in a closed Lie algebra of dimension 4. Therefore, this system is not controllable. In addition, the controllable submanifold of the system in \eqref{eq:ex_so(5)_2} is the integral manifold corresponding to the involutive distribution, $\Delta={\rm span}\{\O_{12}X,\O_{23}X,\O_{13}X,\O_{45}X\}$, passing through the identity matrix $I$.
\end{example}


This simple example illustrates the necessity to compute all possible successive Lie brackets in order to inform uncontrollability of the system. In general, if a system on SO($n$) has a drift and $m$ controls with $m\leq n(n-1)/2$, then one needs to compute Lie brackets up to $(m+1)[(m+1)^{m+1}-1]/2$ times. This can be seen by induction on $k$, the order of iterated Lie brackets, as follows. For $k=1$, the number of the first-order Lie brackets is ${m+1\choose 2}$, since every Lie bracket involves two of the $m+1$ vector fields. For $k=2$, the second-order Lie brackets are the Lie brackets of these ${m+1\choose 2}$ first-order vector fields and the $m+1$ vector fields governing the system dynamics, and hence the number of the second-order Lie brackets is ${m+1\choose 2}(m+1)$. By induction, the total number of Lie brackets needed to be computed for the examination of LARC is $\sum_{i=1}^{m+1}{m+1\choose 2}(m+1)^{i-1}=\sum_{i=1}^{m+1}m(m+1)^i/2=(m + 1)[(m + 1)^{m+1}-1]/2$, and each bracket operation has complexity $O(n^3)$. In addition, applying the LARC also involves the examination of linearly independence, which requires Gaussian elimination with complexity $O(n^3)$ as well. Therefore, the application of LARC is computationally expensive for large $n$ and $m$.

In the following, we present a new notion for analyzing controllability of systems on SO($n$) in terms of the length of permutation cycles on the symmetric group $S_n$.

\subsection{Mapping Lie Bracketing to Permutation Compositions}
To fix the idea, we first consider the driftless system of the form $\dot{X}(t)=\left[\sum_{k=1}^mu_k(t)\Omega_{i_kj_k}\right]X(t)$, $X(0)=I$,
where $\O_{i_kj_k}\in\mathcal{F}=\{\O_{i_1j_1},\O_{i_2j_2},\dots,\O_{i_mj_m}\}\subseteq\mathcal{B}$. 
We will develop a correspondence between the elements of $\mathfrak{so}(n)$ and $S_n$ as well as a mapping between their operations, namely, Lie brackets on $\mathfrak{so}(n)$ and permutations on $S_n$. This nontrivial identification is the key to our new development of controllability conditions.

Recall that every element $\s\in S_n$ is a permutation on $n$ letters, i.e., a bijective map $\s:Z_n\rightarrow Z_n$, where, conventionally, $Z_n=\{1,\dots,n\}$. In addition, an equivalence relation on $Z_n$ can be defined by $a\sim b$ if and only if $b=\s^k(a)$ for $a,b\in Z_n$ and for some $k\in\mathbb{Z}$. The equivalence classes in $Z_n$ determined by this equivalence relation are called the orbits of $\s$. A permutation $\s\in S_n$ is a cycle if it has at most one orbit containing more than one element, and the length of a cycle is the number of elements in its nontrivial orbit. A cycle of length $k$ is also called a $k$-cycle, and, in particular, a 2-cycle is called a transposition. Any permutation of a finite set which contains at least two elements is a product of some transpositions on this set \cite{Lang02}.

Now, let's identify each subset of $\mathcal{B}$ with an element in $S_n$. Let $\mathcal{P}(\mathcal{B})$ denote the power set of $\mathcal{B}$, and define the map 
\begin{equation}
	\label{eq:iota}
	\iota:\mathcal{P}(\mathcal{B})\rightarrow S_n 
\end{equation}
by $\{\O_{i_1j_1},\O_{i_2j_2},\dots,\O_{i_lj_l}\}\mapsto(i_l,j_l)\cdots (i_2,j_2)(i_1,j_1)$, where $(i_k,j_k)$, $k=1,\dots,m$, is the cyclic notation of the following permutation,
\begin{align*}
	\left(\begin{array}{ccccccc} 1 & \cdots & i_k & \cdots & j_k & \cdots & n \\ 1 & \cdots & j_k & \cdots & i_k & \cdots & n \end{array}\right).
\end{align*}
We note that $\iota$ is a map that is not a well-defined function.

\begin{lemma}
	\label{lem:iota}
	The map $\iota:\mathcal{P}(\mathcal{B})\rightarrow S_n$ is surjective.
\end{lemma}
\begin{proof}
	Because any permutation can be expressed as a product of transpositions, then for any $\s=(i_l,j_l)\cdots(i_1,j_1)\in S_n$, there exists a subset $\mathcal{S}=\{\O_{i_1j_1},\dots,\O_{i_lj_l}\}\subseteq\B$ such that $\iota(\mathcal{S})=\s$.
\end{proof}

\begin{remark}[\textbf{The notion of bridging index}]
	\rm Lemma \ref{lem:iota} reveals that every permutation can be associated with a subset of $\B$ through the map $\iota$, which allows us to characterize the relationship between the Lie bracket operation on $\mathfrak{so}(n)$ and the permutation operation on $S_n$ under $\iota$. Consider $\S=\{\O_{ij},\O_{kl}\}$, then we have $[\O_{ij},\O_{kl}]=\O_{il}$ if $j=k$, where $\O_{il}$ is distinct from and linearly independent of the elements in $\S$. 
Applying $\iota$ to the set $\S$ gives $\iota(\S)=\iota(\O_{kl})\iota(\O_{ij})=\iota(\O_{jl})\iota(\O_{ij})=(j,l)(i,j)=(i,j,l)$, which is a cycle of length 3. The increase of the cycle length by 1 (transpositions have length 2) is due to 
the bridging index $j=k$. On the other hand, if $[\O_{ij},\O_{kl}]=0$, then there are two cases: (i) $i=k$ and $j=l$, so that $(i,j)(k,l)=e$, where $e\in{S_n}$ is the identity map on $Z_n$; and (ii) $i,j,k,l$ are all distinct, then $(i,j)(k,l)=(k,l)(i,j)$ is a permutation composed of the product of two disjoint transpositions. Note that case (i) and (ii) represent the commutativity property of the group actions over $S_n$, which corresponds to the vanishing of Lie brackets on $\mathfrak{so}(n)$; whereas nonvanishing of Lie brackets leads to the increase of the cycle length. 

Inductively, for $\iota(\S)=(i,j,l)$ with its index set denoted $J=\{i,j,l\}$ and for some $\O_{ab}\in\B$, we have $[\O_{ab},\O_{pq}]=0$ if $a,b\not\in J$ and $p,q\in J$, and in this case, $\iota(\O_{ab})\iota(\mathcal{S})=(a,b)(i,j,l)$ is a permutation as a product of two disjoint cycles. However, if either $a\in J$ or $b\in J$, then we have $[\O_{ab},\O_{pq}]\in\B\backslash\S$ for any $p,q\in J$; also $\iota(\O_{ab})\iota(\mathcal{S})$ must be a 4-cycle with the nontrivial orbit $\{i,j,l,a\}$ assuming $b\in J$ which serves as the bridging index.
\end{remark}

Now, we revisit Examples \ref{ex:so(5)} and \ref{ex:so(5)_2} to illustrate 
the application of $\iota$ for mapping Lie brackets to permutations. Meanwhile, we use these examples to motivate the idea of interpreting controllability in terms of the length of permutation cycles. 

\begin{example} \rm
	\label{ex:so(5)_i}
	Recall the system in \eqref{eq:ex_so(5)} in Example \ref{ex:so(5)} with the set of control vector fields $\mathcal{F}_1=\{\O_{i,i+1}:1\leq i\leq4\}$. 
	Applying the map $\iota$, the successive Lie brackets 
are mapped to permutation products,
\begin{align*}
	[\O_{12},\O_{23}]=\O_{13} &\mapsto (1,2)(2,3)=(1,2,3),\\
	[\O_{23},\O_{34}]=\O_{24} &\mapsto (2,3)(3,4)=(2,3,4),\\
	[\O_{34},\O_{45}]=\O_{35} &\mapsto (3,4)(4,5)=(3,4,5),\\
	[\O_{12},[\O_{23},\O_{34}]]=\O_{14} &\mapsto (1,2)(2,3,4)=(1,2,3,4),\\
	[\O_{23},[\O_{34},\O_{45}]]=\O_{25} &\mapsto (2,3)(3,4,5)=(2,3,4,5),\\
	[\O_{12},[\O_{23},[\O_{34},\O_{45}]]]=\O_{15} &\mapsto (1,2)(2,3,4,5)\\
	&\quad \ = (1,2,3,4,5).
\end{align*}
At each iteration, the resulting Lie bracket is nontrivial and distinct, and thus 
the corresponding permutation is a cycle with increased length as shown above. 
In addition, $\iota(\mathcal{F}_1)=(1,2)(2,3)(3,4)(4,5)=(1,2,3,4,5)$ 
is a cycle of length 5, that is, the cycle of maximum length in $S_5$. This suggests that controllability of a system defined on the special orthogonal Lie group may be determined by the length of permutation cycles on the associated symmetric group, because the system in \eqref{eq:ex_so(5)} was shown to be controllable in Example \ref{ex:so(5)}.

This conjecture can be further verified using Example \ref{ex:so(5)_2}, where the system in \eqref{eq:ex_so(5)_2} is not controllable. In this case, the given control vector fields 
are denoted by $\mathcal{F}_2=\{\O_{12},\O_{23},\O_{45}\}$.  Because a 5-cycle can be decomposed as a product of at least 4 transpositions, $\iota(\mathcal{F}_2)$ cannot be a 5-cycle. This suggests that the system is not controllable. In fact, $\iota(\mathcal{F}_2)=(1,2)(2,3)(4,5)=(1,2,3)(4,5)$ is a permutation composed of a product of two disjoint cycles with nontrivial orbits $\{1,2,3\}$ and $\{4,5\}$, respectively. Note the reason that $\iota(\mathcal{F}_2)$ is not a cycle of length 5 is due to the lack of bridging transpositions that transport an element in the orbit $\{1,2,3\}$ to an element in the orbit $\{4,5\}$, which results in the uncontrollability of the system in \eqref{eq:ex_so(5)_2}. 
\end{example}

Next, we will rigorously demonstrate the use of the length of permutation cycles for characterizing controllability of the system defined on SO$(n)$ for $n\geq 3$.

\subsection{Controllability in terms of Length of Permutation Cycles}
Example \ref{ex:so(5)_i} sheds light on determining controllability for systems on SO$(n)$ in terms of the length of permutation cycles on $S_n$. In this section, we prove this nontrivial observation.

\begin{theorem}
	\label{thm:son}
	The control system defined on {\rm SO}$(n)$ of the form
	\begin{align}
		\label{eq:son}
	     \dot{X}(t)=\left[\sum_{k=1}^mu_k(t)\Omega_{i_kj_k}\right]X(t), \quad X(0)=I,
	\end{align}
	where $\O_{i_kj_k}\in\F=\{\O_{i_1j_1},\dots,\O_{i_mj_m}\}\subseteq\B$, with $1\leq i_k<j_k\leq n$ for $k=1,\dots,m$, are elements of the standard basis of $\mathfrak{so}(n)$, is controllable if and only if there is a subset $\mathcal{S}\subseteq\F$ such that $\iota(\mathcal{S})$ is an $n$-cycle, where $\iota$ is the map defined in \eqref{eq:iota}.
\end{theorem}

\begin{proof} 
	We know, by the LARC, that the system in \eqref{eq:son} is controllable on SO$(n)$ if and only if ${\rm Lie}(\mathcal{F})=\mathfrak{so}(n)$. Therefore, it is equivalent to showing that ${\rm Lie}(\mathcal{S})=\mathfrak{so}(n)$ if and only if $\iota(\mathcal{S})$ is an $n$-cycle for some $\mathcal{S}\subseteq\mathcal{F}$.
	
	(Sufficiency) Suppose there exists a subset $\S\subseteq\F$ such that $\iota(\mathcal{S})$ is an $n$-cycle. Because an $n$-cycle can be decomposed into a product of at least $n-1$ transpositions, this implies $m\geq n-1$. Hence, it suffices to assume that the cardinality of $\mathcal{S}$ is $n-1$, and, without loss of generality, let $\mathcal{S}=\{\O_{i_1j_1},\dots,\O_{i_{n-1}j_{n-1}}\}$. Because $\iota(\mathcal{S})$ is an $n$-cycle, it follows that the index set $\{i_1,j_1,\dots,i_{n-1},j_{n-1}\}=\{1,\ldots,n\}$. 
	Note that the set $\{i_1,j_1,\dots,i_{n-1},j_{n-1}\}$ may contain repeated elements.


($n=3$): Suppose there exists a subset $\S=\{\O_{ij},\O_{kl}\}\subset\F$ and $\iota(\S)=(i,j)(k,l)$ is a 3-cycle, i.e., it must be $i=k$, $j=k$ $i=l$, or $j=l$. Then, we have $[\O_{ij},\O_{kl}]\in\B\backslash\S$. Therefore, ${\rm span}\{\O_{ij},\O_{kl},[\O_{ij},\O_{kl}]\}=\mathfrak{so}(3)$, and the system in \eqref{eq:son} is controllable on SO$(3)$.

Now, assume that a system defined on SO$(n-1)$, $n\geq4$, in the form of \eqref{eq:son} is controllable if there is $\S\subseteq\F$ such that $\iota(\S)$ is an $(n-1)$-cycle. 
Let $\S\subseteq\F$ be a set of $n-1$ elements such that $\iota(\S)=(i_{n-1},j_{n-1})(i_{n-2},j_{n-2})\cdots(i_1,j_1)$ is a cycle of length $n$, then for every $k=1,\dots,n-1$, there exists some $l=1,\dots,n-1$ such that $\{i_k,j_k\}\cap\{i_l,j_l\}\neq\varnothing$. Consequently, there are $n-2$ of the transpositions $(i_k,j_k)$, $k=1,\dots,n-1$, such that their multiplication is a cycle of length $n-1$. Without loss of generality, assume that $\iota(\S\backslash\{\O_{i_{n-1},j_{n-1}}\})=(i_{n-2},j_{n-2})\cdots(i_1,j_1)$ is a $(n-1)$-cycle with the nontrivial orbit $\{i_1,j_1,\dots,i_{n-2},j_{n-2}\}=\{1,\dots,n-1\}$. By the induction hypothesis, the system in \eqref{eq:son} is controllable on SO$(n-1)\subset$ SO$(n)$. Equivalently, any $\O_{ij}\in\B$ such that $1\leq i<j\leq n-1$ can be generated by iterated Lie brackets of the elements in $\S\backslash\{\O_{i_{n-1},j_{n-1}}\}$. Because $\iota(\S)=(i_{n-1},j_{n-1})\iota(\S\backslash\{\O_{i_{n-1},j_{n-1}}\})$ is a $n$-cycle, we must have $i_{n-1}\in\{1,\dots,n-1\}$ and $j_{n-1}=n$. Therefore, $\O_{kn}$ can be generated by the Lie brackets $[\O_{ki_{n-1}},\O_{i_{n-1}j_{n-1}}]$ for any $k=1,\dots,n-1$. Consequently, the system in \eqref{eq:son} is controllable on SO$(n)$.

(Necessity) Because the system in \eqref{eq:son} is controllable, ${\rm Lie}(\mathcal{F})=\mathfrak{so}(n)$. Then, there exists a subset $\mathcal{S}$ of $\mathcal{F}$ such that ${\rm Lie}(\mathcal{S})=\mathfrak{so}(n)$ and $\S$ contains no \emph{redundant elements}, i.e., the elements that can be generated by Lie brackets of the other elements in $\mathcal{S}$. Without loss of generality, we assume $\mathcal{S}=\{\O_{i_1j_1},\dots,\O_{i_lj_l}\}$, where $l\leq m$. By Lemma \ref{lem:son}, for any $\O_{ab},\O_{cd}\in\S$, if $[\O_{ab},\O_{cd}]\neq0$, then there must exist a bridging index, i.e., must be one of the cases of $a=c$, $a=d$, $b=c$, or $b=d$. This, together with $\Lie(\S)=\mathfrak{so}(n)$, implies that the index set $J$ of $\S$ is $J=\{i_1,j_1,\dots,i_l,j_l\}=\{1,\dots,n\}$, and for any $\O_{i_kj_k}\in\S$, there exists some $\O_{i_sj_s}\in \S$ with $s\neq k$ such that $\{i_k,j_k\}\cap\{i_s,j_s\}\neq\varnothing$. Moreover, because $\S$ does not contain redundant elements, $\iota(\mathcal{S})=\iota(\O_{i_lj_l})\cdots\iota(\O_{i_1j_1})$ is a cycle whose orbit contains every element in $\{1,\dots,n\}$, namely, it is a cycle of length $n$. In addition, the cardinality of $\mathcal{S}$ is $n-1$.
\end{proof}

\begin{remark}
	\label{rem:number_of_control}
	\rm Following the proof of Theorem \ref{thm:son}, it requires at least $n-1$ controls for the system on SO$(n)$ as in \eqref{eq:son} to be fully controllable and, on the other hand, for $\iota(\mathcal{S})$, $\S\subseteq\F$, to reach a cycle of length $n$. 
\end{remark}

\begin{corollary}
	\label{cor:son}
	The controllable submanifold of the system in \eqref{eq:son} is determined by the orbits of $\iota(\mathcal{S})$, where $\mathcal{S}\subseteq\F$ 
	satisfies ${\rm Lie}(\S)={\rm Lie}(\F)$.
\end{corollary}

\begin{proof}  
	Let $\mathcal{S}$ be a subset of $\F$ such that $\Lie(\S)=\Lie(\F)$ and $\S$ does not contain redundant elements. First, let $\s=\iota(\S)\in S_n$ be a cycle with nontrivial orbit $\mathcal{O}$, 
then Theorem \ref{thm:son} implies ${\rm Lie}(\mathcal{S})={\rm span}\{\Omega_{ij}:i,j\in\mathcal{O},i<j\}$. Next, if $\s=\s_1\cdots\s_l$ is a permutation as a product of disjoint cycles $\s_1,\dots,\s_l$ with $l\geq2$, then there exists a partition $\{\mathcal{S}_1,\dots,\mathcal{S}_l\}$ of $\mathcal{S}$ such that $\iota(\mathcal{S}_k)=\s_k$ for each $k=1,\dots,l$. Let $\mathcal{O}_k$ denotes the nontrivial orbit of $\s_k$ for each $k=1,\dots,l$, then $\Lie(\S_k)=\{\O_{ij}:i,j\in\mathcal{O}_k,i<j\}$ and the sets $\mathcal{O}_1,\dots,\mathcal{O}_l$ are pairwise disjoint subsets of $\{1,\dots,n\}$. Hence, $\Lie(\mathcal{S}_i)\cap\Lie(\mathcal{S}_j)=\{0\}$ holds for all $i\neq j$, and consequently, we have ${\rm Lie}(\mathcal{S})={\rm Lie}(\mathcal{S}_1)\oplus\cdots\oplus{\rm Lie}\mathcal({S}_l)$, where $\oplus$ denotes the direct sum of vector spaces. By Frobenius theorem \cite{Lee03}, ${\rm Lie}(\mathcal{S})$ is completely integrable, and the set of all its maximal integral manifolds forms a foliation $F$ of SO$(n)$. Since the initial condition of the system in \eqref{eq:son} is the identity matrix $I$, the leaf of $F$ passing through $I$ is the controllable submanifold of the system in \eqref{eq:son}.
\end{proof}

According to Theorem  \ref{thm:son} and Corollary \ref{cor:son}, mapping the control vector fields in $\F$ to permutations provides not only an alternative approach to effectively examine controllability of systems defined on SO($n$), but also a systematic procedure to characterize the controllable submanifold when the system is not fully controllable.

\begin{example}[Controllable Submanifold]
	\label{ex:controllable_submanifold}
	\rm	Recall Example \ref{ex:so(5)_2} where the system in \eqref{eq:ex_so(5)_2} is not controllable and there exist no subsets of $\F=\{\O_{12},\O_{23},\O_{45}\}$ such that $\iota(\F)$ is a 5-cycle. In addition, the controllable submanifold is the integral manifold of the involutive distribution $\Delta={\rm span}\{\O_{12}X,\O_{23}X,\O_{13}X,\O_{45}X\}=\{\O_{ij}X:i,j\in\{1,2,3\}\text{ or }i,j\in\{4,5\}\}$. This can be identified by the nontrivial orbits of $\iota(\F)=(1,2,3)(4,5)$. On the other hand, for each $X\in$ SO(5), the complement $\Delta_X^{\perp}={\rm span}\{\O_{ij}X:i\in\{1,2,3\},\ j\in\{4,5\}\}$ of the distribution evaluated at $X$ contains the bridging elements required for full controllability of this system. 
\end{example}

\section{Interpreting Controllability of Systems on SO$(n)$ through a Monoid Structure on $S_n$}
\label{sec:*}
From Theorem \ref{thm:son}, the existence of a subset $\S\subseteq\F$ with $\iota(\S)$ an $n$-cycle in $S_n$ determines controllability of the system on SO$(n)$. Checking this condition, in general, may be highly combinatorial, 
because by Remark \ref{rem:number_of_control} it requires at least $n-1$ controls for this system to be controllable, and there are $m\choose n-1$ subsets of $\F$ consisting of $n-1$ elements for $m>n-1$. On the other hand, in $S_n$ every transposition is its own inverse by the cancellation law, i.e., $(p,q)^{-1}=(p,q)$ for $p,q\in\{1,\dots,n\}$. This group operation may result in a decreased cycle length under the action of the map $\iota$, defined in \eqref{eq:iota}, because $\mathcal{P}(\B)$ is not equipped with a group structure that provides each element an inverse. This issue is illustrated in the following example.

\begin{example}
	\label{ex:redundant}
	\rm Consider the system on SO(4), given by
	\begin{align}
		\label{eq:redundant}
		\dot{X}(t) &= \left[u_1(t)\O_{12}+u_2(t)\O_{23}+u_3(t)\O_{13}+u_4\O_{34}\right]X(t),\nonumber\\
		X(0) &= I.
	\end{align}
	For $\mathcal{S}=\{\O_{12},\O_{23},\O_{34}\}\subset\F=\{\O_{12},\O_{23},\O_{13},\O_{34}\}$, we obtain $\iota(\mathcal{S})=(3,4)(2,3)(1,2)=(1,4,3,2)$, which is a 4-cycle in $S_4$ and which implies controllability of the system in \eqref{eq:redundant} on SO(4) by Theorem \ref{thm:son}. However, for $\S'=\F$, then $\iota(\S')=(3,4)(1,3)(2,3)(1,2)=(3,4)(1,3)(1,3,2)=(3,4)(1,3)^{-1}(1,3)(2,3)=(3,4)(2,3)=(2,3,4)$ is a 3-cycle, in spite of Lie$(\S')=\mathfrak{so}(4)$. This is due to the lack of the inverse operation on $\mathcal{P}(\B)$, while there is a cancellation law in $S_n$. As a result, the \emph{redundant} basis $\O_{13}$, in the sense that it can be generated by iterated Lie brackets of other elements in $\S'$, i.e., 
$[\O_{12},\O_{23}]=\O_{13}$, leads to a degenerate case, that is, the decrease of the cycle length following the permutation operations under the map $\iota$.
\end{example}


\begin{remark}[\textbf{Degeneracy under the $\iota$ Map}]
	\label{rem:degeneracy}
	\rm	According to Corollary \ref{cor:son}, if $\S\subset\mathcal{B}$ and $\iota(\mathcal{S})=(a_1,\dots,a_k)$, $k<n$, is a $k$-cycle in $S_n$, then ${\rm Lie}(\mathcal{S})={\rm span}\{\O_{ij}\in\mathcal{B}:i,j=a_1,\dots,a_k\}$. For any $\O_{a_s a_t}\in{\rm Lie}(\mathcal{S})\cap\mathcal{B}$, we have 
\begin{align}
&\iota(\O_{a_s a_t})\iota(\mathcal{S})=(a_s,a_t)(a_1,\dots,a_k)\nonumber\\
&=\begin{cases}
(a_1,\dots,a_{t-1})(a_t,\dots,a_k),\ \text{if}\ s=1,\\
(a_2,\dots,a_{s-1},a_k)(a_s,\dots,a_{k-1}), \ \text{if}\ s\neq1\ \text{and}\ t=k,\\
(a_1,\dots,a_{s-1},a_t,\dots,a_k)(a_s,\dots,a_{t-1}),\ \text{otherwise},
\end{cases} \label{eq:redundant_calculation}
\end{align}
which are not $k$-cycles. 
Specifically, if $(s,t)$ is equal to $(1,2)$, $(1,k)$, or $(k-1,k)$, then $\iota(\O_{a_s,a_t})\iota(\mathcal{S})$ is a $(k-1)$-cycle; otherwise, it is a permutation as a product of two disjoint cycles of length $n-s+t$ and $t-s$, respectively. In summary, if $\iota(\mathcal{S})$ is an $l$-cycle and $\mathcal{S}'$ is a subset of $\mathcal{B}$ containing  $l'$ elements such that ${\rm Lie}(\mathcal{S}')\cap{\rm Lie}(\mathcal{S})\neq\varnothing$, then $\iota(\mathcal{S}')\iota(\mathcal{S})$ is a cycle of length no greater than $l+l'-1$ or a permutation as a product of disjoint cycles.
\end{remark}

A proper modification of the permutation multiplication can be made to deal with such degenerate situations, specifically, by redefining the 
binary operation on $S_n$. In the following sections, we will introduce an equivalence relation, compatible with an alternative binary operation on $S_n$, so that controllability of the system on SO$(n)$ can be determined directly based on the entire set of the control vector fields $\F$.

\subsection{Equivalence Relation on $S_n$}
\label{sec:equivalence}
Because the symmetric group is non-abelian, the map $\iota:\mathcal{P}(\mathcal{B})\rightarrow S_n$, defined in \eqref{eq:iota}, is not a well-defined function. As a result, for two subsets $\S,\S'\subset\F$, $\iota(\S)$ and $\iota(\S')$ may be different permutations sharing the same orbits. For example, $\S=\{\O_{12},\O_{23}\}$ and $\S'=\{\O_{23},\O_{12}\}$ are identical sets, but $\iota(\S)=(2,3)(1,2)=(1,3,2)$ and $\iota(\S')=(1,2)(2,3)=(1,2,3)$ are different permutations with the same orbit. 
In this situation, they characterize identical controllable submanifold, which in turn motivates the need to introduce an equivalence relation on $S_n$.

For any $\s,\eta\in S_n$, we define the equivalence relation $\sim$ between them, and say $\s\sim\eta$ if and only if they have the same orbits. It is straightforward to check the transitivity, reflexivity, and symmetry of $\sim$. Let $S_n/\sim$ denote the set of equivalent classes in $S_n$, and $[\s]\in S_n/\sim$ denote the class of all permutations in $S_n$ with the same orbits as $\s$, i.e., $[\s]=\{\eta\in S_n:\eta\sim\s\}$. If $\s_1,\dots,\s_k$ are $k$ transpositions in $S_n$ such that $\s_k\cdots\s_1$ is a $(k+1)$-cycle and $\eta\in S_k$ is a permutation on $k$ letters, then $\s_{\eta(k)}\cdots\s_{\eta(1)}$ is also a $(k+1)$-cycle that has the same nontrivial orbit as $\s_k\cdots\s_1$. Using the equivalence relation, the map $\iota$ in \eqref{eq:iota} can be redefined as a well-defined function, that is, $\iota:\mathcal{P}(\B)\to S_n/\sim$, by $\{\O_{i_1,j_1},\dots,\O_{i_m,j_m}\}\mapsto[(i_1,j_1)\cdots(i_m,j_m)]$. However, the equivalence relation is not necessarily compatible with the group operation on $S_n$, and thus we introduce a 
binary operation on $S_n$ that offers compatibility.

\subsection{Monoid Structure on $S_n$}
\label{sec:monoid}
Here, we introduce a binary operation $*$ for transpositions in $S_n$ by 
\begin{align}
	\label{eq:*}
	\s*\eta=
	\left\{\begin{array}{ll}
	\s, & \text{if\ } \s=\eta, \\
	\s\cdot\eta, & \text{otherwise},\end{array}\right.
\end{align}
where $\s,\eta\in S_n$ are transpositions and `$\cdot$' is the group operation on $S_n$. Because every permutation in $S_n$ is a product of transpositions, the $*$ operation is applicable to any permutations. Under this operation, 
every transposition is idempotent, and hence every element of $S_n$ has no inverse except for the identity element $e$, which implies that $(S_n,*)$ is not a group. As we know, any cycle $\s$ of length $m\leq n$ in the symmetric group $(S_n,\cdot)$ can be represented as a product of at least $m-1$ distinct transpositions, 
say $\s=\s_{m-1}\cdots\s_2\cdot\s_1$, so that $\s=\s_{m-1}\cdots\s_1=\s_{m-1}*\cdots*\s_1$ by the definition of $*$ in \eqref{eq:*}. Furthermore, because every permutation $\eta$ in $(S_n,\cdot)$ is a product of finitely many disjoint cycles, then we have $\eta=c_k\cdots c_1=c_k*\cdots *c_1$ 
for some disjoint cycles $c_1,\dots,c_k$. 

Next, we illustrate the computation of the $*$ operation on $S_n$. Suppose that $\s_1\in S_n$ is a transposition and $\s_2\in S_n$ is a cycle, and let $\mathcal{O}_1,\mathcal{O}_2\subseteq\{1,\dots,n\}$ denote their nontrivial orbits, respectively. If $\mathcal{O}_1\not\subset\mathcal{O}_2$,
then 
$\s_2*\s_1=\s_2\cdot\s_1$; otherwise, if $\mathcal{O}_1\subset\mathcal{O}_2$, then $\s_2=\eta\cdot\s_1$ for some $\eta\in S_n$, which implies that $\s_2*\s_1=\eta\cdot\s_1*\s_1=\eta\cdot\s_1=\s_2$. Because every transposition is its own inverse under the `$\cdot$' operation, this gives $\eta=\s_2\cdot\s_1^{-1}=\s_2\cdot\s_1$. Similarly, if $\s_1$ is a cycle and $\s_2$ is a transposition, it follows that $\s_2*\s_1=\s_2\cdot\s_1$ if $\mathcal{O}_2\not\subset\mathcal{O}_1$, and $\s_2*\s_1=\s_1$ otherwise. Moreover, for any two cycles $\s_1,\s_2\in S_n$, if the cardinality $|\mathcal{O}_1\cap\mathcal{O}_2|<2$, then
$\s_2*\s_1=\s_2\cdot\s_1$ because the minimal decompositions of $\s_1$ and $\s_2$ do not share a common transposition. Otherwise, if $|\mathcal{O}_1\cap\mathcal{O}_2|\geq 2$, pick $i,j\in\mathcal{O}_1\cap\mathcal{O}_2$ and define $\tau_1=(i,j)$, then we have $\s_1=\tau_1\cdot\eta_{1}=\tau_1*\eta_{1}$ and $\s_2=\xi_{1}\cdot\tau_1=\xi_{1}*\tau_1$ for some $\eta_{1},\xi_{1}\in S_n$. It follows that
\begin{align}
	\label{eq:*operation_r}
	\s_2*\s_1&=(\xi_{1}*\tau_1)*(\tau_1*\eta_{1})=\xi_{1}*\tau_1*\tau_1*\eta_{1}\nonumber\\
	&= \xi_{1}*\tau_1*\eta_{1}=\xi_{1}*\tau_1\cdot\eta_{1}=\xi_{1}*\s_1, \\
	\label{eq:*operation_l}
	&\qquad\qquad\qquad\qquad\qquad\quad=\s_2*\eta_{1}.
\end{align}
The calculation of the $*$ operation for general permutations follows the same argument as for cycles shown in \eqref{eq:*operation_r} and \eqref{eq:*operation_l}. Note that although $\s_2*\s_1$ may lead to different results as presented in \eqref{eq:*operation_r} and \eqref{eq:*operation_l}, they are equivalent, i.e., resulting in identical orbits, under the $*$ operation over the coset $S_n/\sim$ (see Lemma \ref{lem:*orbit} in Appendix). For instance, consider two permutations in $S_3$, $\s_1=(1,2,3)$ and $\s_2=(1,3,2)$, we have $\s_2*\s_1=\s_1=(1,2,3)\in[(1,2,3)]$ by \eqref{eq:*operation_r} and $\s_2*\s_1=\s_2=(1,3,2)\in[(1,2,3)]$ by \eqref{eq:*operation_l}.

\begin{theorem}
	\label{thm:monoid}
	$(S_n/\sim,*)$ is a commutative monoid.
\end{theorem}
\begin{proof}
	The proof of this theorem is based on the associativity, invariance, and commutativity of the $*$ operation over the equivalence classes in $S_n$, shown in Lemmas \ref{lem:*orbit} and \ref{lem:*associativity} and Corollaries \ref{cor:*orbit} and \ref{cor:*commutativity} in Appendix. Moreover, the identify element $e$ of the symmetric group $(S_n,\cdot)$ is unique, so that $\s*e=\s\cdot e=\s=e\cdot\s=e*\s$ holds for any $\s\in S_n$. In addition, because $e$ is the only element without nontrivial orbit in $S_n$, we have $[e]=\{e\}$, which leads to $[\s]*[e]=[\s*e]=[e*\s]=[e]*[\s]=[\s]$ for any $[\s]\in S_n/\sim$. Note that in the terms $[\s]*[e]$ and $[e]*[\s]$ above, $*$ denotes the operation on the equivalence classes induced by the $*$ operation on permutations. Therefore, $[e]$ is the identity element of $(S_n/\sim,*)$, and, together with the associativity and commutativity of $*$ on $S_n/\sim$, $(S_n/\sim,*)$ is a commutative monoid.
\end{proof}

With the algebraic structures defined in Section \ref{sec:equivalence} and \ref{sec:monoid}, we will characterize controllability of systems defined on SO($n$) over the monoid $(S_n/\sim,*)$.

\subsection{Controllability Characterization over the Monoid Structure}
\label{sec:controllability_monoid}
Recall that $\iota:\mathcal{P}(\mathcal{B})\rightarrow S_n$ defined in \eqref{eq:iota} in Section \ref{sec:SO(n)_controllability} by $\{\O_{i_1j_1}\dots,\O_{i_mj_m}\}\mapsto(i_m,j_m)\cdots(i_1,j_1)$ maps standard basis vector fields in $\mathfrak{so}(n)$ to transpositions in $S_n$ 
and that $\iota$ is not a well-defined function on the power set $\mathcal{P}(\mathcal{B})$. To study controllability over a monoid structure in $S_n$, we modify the map $\iota$ by lifting its range from $(S_n,\cdot)$ to $(S_n/\sim,*)$ and define $\~{\iota}:(\mathcal{P}(\mathcal{B}),\cup)\rightarrow(S_n/\sim,*)$ by $\{\O_{i_1j_1},\dots,\O_{i_kj_k}\}\mapsto[(i_1,j_1)*\cdots*(i_k,j_k)]$, where $(\mathcal{P}(\mathcal{B}),\cup)$ is a commutative monoid because the union $\cup:\mathcal{P}(\mathcal{B})\times\mathcal{P}(\mathcal{B})\rightarrow\mathcal{P}(\mathcal{B})$ is a binary operation on $\mathcal{P}(\mathcal{B})$ with the empty set $\varnothing$ as the identity element of $(\mathcal{P}(\mathcal{B}),\cup)$, and is associative and commutative.


\begin{lemma}
	The map $\~{\iota}:(\mathcal{P}(\mathcal{B}),\cup)\rightarrow(S_n/\sim,*)$ defined by $\{\O_{i_1j_1},\dots,\O_{i_kj_k}\}\mapsto[(i_1,j_1)*\cdots*(i_k,j_k)]$ is a monoid homomorphism.
\end{lemma}
\begin{proof}
We start with proving that $\~{\iota}$ is well-defined. Consider two subsets, $\S$ and $\S'$, of $k$ identical elements in $\mathcal{P}(\mathcal{B})$, say $\mathcal{S}=\{\O_{i_1j_1},\dots,\O_{i_kj_k}\}$ and $\mathcal{S}'=\{\O_{i_1'j_1'},\dots,\O_{i_k'j_k'}\}$, where $i_s'=i_{\s(t)}$ and $j_s'=j_{\s(t)}$, $s,t=1,\dots,k$, for some $\s\in S_k$. By the definition of $\~{\iota}$ 
and the compatibility of $*$ with $\sim$ from Corollary \ref{cor:*orbit}, we have $\~{\iota}(\S)=[(i_1,j_1)*\cdots*(i_k,j_k)]=[(i_1,j_1)]*\cdots*[(i_k,j_k)]$ and $\~{\iota}(\S')=[(i_1',j_1')*\cdots*(i_k',j_k')]=[(i_1',j_1')]*\cdots*[(i_k',j_k')]=[(i_{\s(1)},j_{\s(1)})]*\cdots*[(i_{\s(k)},j_{\s(k)})]$. Since $(S_n/\sim,*)$ is commutative, $[(i_1,j_1)]*\cdots*[(i_k,j_k)]=[(i_{\s(1)},j_{\s(1)})]*\cdots*[(i_{\s(k)},j_{\s(k)})]$, and thus $\~{\iota}(\mathcal{S})=\~{\iota}(\mathcal{S}')$, which implies that $\~{\iota}$ is well-defined. In addition, if $\~{\iota}(\S)=[\s]$ and $\~{\iota}(\S')=[\s']$, then it holds that $\~{\iota}(\S\cup\S')=[\s*\s']=[\s]*[\s']=\~{\iota}(\S)*\~{\iota}(\S')$. This, together with $\tilde\iota(\varnothing)=[e]$, concludes that $\tilde{\iota}$ is a monoid homomorphism.
\end{proof}

Now, analogous to the developments in Section \ref{sec:SO(n)_controllability}, we explore the correspondance between Lie bracket operations on $\mathcal{B}$ and the $*$ operation on $S_n/\sim$, which will facilitate the controllability analysis. The basic idea is illuminated by the following example.

\begin{example} \rm
	\label{ex:redundant_*}
	Recall the system on SO($4$) in \eqref{eq:redundant} in Example \ref{ex:redundant}, where $\mathcal{F}=\{\O_{12},\O_{23},\O_{13},\O_{34}\}$ and the system is controllable on SO(4). However, $\iota(\F)=(2,3,4)$ is of length $3<4$, which does not report controllability. This is due to the existence of redundant elements, e.g., $\O_{13}=[\O_{12},\O_{23}]$ or $\O_{23}=[\O_{13},\O_{12}]$, so that the degeneracy occurs following the composition of the $\iota$ operations, which results in the reduced cycle length. On the other hand, on the monoid $(S_4/\sim,*)$, we obtain $\~{\iota}(\mathcal{F})=[(1,2)]*[(2,3)]*[(1,3)]*[(3,4)]=[(1,2)*(2,3)]*[(1,3)]*[(3,4)]=[(1,2,3)]*[(1,3)]*[(3,4)]=[(1,3)]*[(1,2)]*[(1,3)]*[(3,4)]=[(1,2)]*[(1,3)]*[(1,3)]*[(3,4)]=[(1,2)]*[(1,3)]*[(3,4)]=[(1,2,3,4)]$, which is the equivalent class of 4-cycles in $S_4$.
\end{example}

\begin{remark}[\textbf{Nondegeneracy under the $\~\iota$ Map}] 
\label{rmk:redundant}
	\rm	
	Let $\S\subset\mathcal{B}$ be a set containing $k-1$ elements such that $\~\iota(\S)=[\s]=[(a_1,\dots,a_k)]$, i.e., $\S$ does not have redundant elements, then $\iota(\S)=\eta$ holds for some $\eta\in[\s]$, and, consequently, ${\rm Lie}(\mathcal{S})={\rm span}\{\O_{ij}\in\mathcal{B}:i,j=a_1,\dots,a_k\}$ by Corollary \ref{cor:son}. For any $\O_{ij}\in\mathcal{B}$, there are two possibilities (i) if $\O_{ij}\not\in{\rm Lie}(\mathcal{S})$, then $\{i,j\}\not\subseteq\{a_1,\dots,a_k\}$ and thus $\~{\iota}(\O_{ij})*\~{\iota}(\mathcal{S})=[(i,j)]*[(a_1,\dots,a_k)]=[(i,j)*(a_1,\dots,a_k)]=[(i,j)\cdot(a_1,\dots,a_k)]$. In this case, if $\{i,j\}\cap\{a_1,\dots,a_k\}\neq\varnothing$, say $\{i,j\}\cap\{a_1,\dots,a_k\}=\{j\}$, then $\~{\iota}(\O_{ij})*\~{\iota}(\mathcal{S})=[(i,a_1,\dots,a_k)]$ is the equivalence class of $(k+1)$-cycles with the nontrivial orbit $\{i,a_1,\dots,a_k\}$; otherwise, if $\{i,j\}\cap\{a_1,\dots,a_k\}=\varnothing$, then $\~\iota(\O_{ij})*\~\iota(\mathcal{S})=[(i,j)\cdot(a_1,\dots,a_k)]$ is the equivalence class of permutations with two nontrivial orbits $\{i,j\}$ and $\{a_1,\dots,a_k\}$; (ii) if $\O_{ij}\in{\rm Lie}(\mathcal{S})$, i.e., $\O_{ij}$ is redundant, then $\{i,j\}\subseteq\{a_1,\dots,a_k\}$, and there exists a unique $\xi\in S_n$ such that $\eta=(i,j)\cdot\xi=(i,j)*\xi$. Hence, $\~{\iota}(\O_{ij})*\~{\iota}(\mathcal{S})=[(i,j)]*[(i,j)*\xi]=[(i,j)*(i,j)*\xi]=[(i,j)*\xi]=[(i,j)\cdot\xi]=[\eta]=[\s]$, which implies that $\~{\iota}(\S\cup\{\O_{ij}\})=\~{\iota}(\S)=[\s]$. This shows the invariance of the image of $\~\iota$ when acting on a subset $\S\subset\B$ consisting of redundant elements. 
\end{remark}

Example \ref{ex:redundant_*} and Remark \ref{rmk:redundant} illustrate that on the monoid $(S_n/\sim,*)$, $\tilde{\iota}(\mathcal{F})$ results in a class of permutations containing cycles of maximum possible length. Therefore, controllability of the system in \eqref{eq:son} can be examined directly using the entire set of control vector fields $\F$ through $\~\iota$ over $(S_n/\sim,*)$.

\begin{theorem}
	\label{thm:*operation}
	The control system defined on ${\rm SO}(n)$ as in \eqref{eq:son} is controllable if and only if $\tilde{\iota}(\mathcal{F})=[(1,\dots,n)]$ in $(S_n/\sim,*)$.
\end{theorem}
\begin{proof}
The proof of the sufficiency is identical to 
that of Theorem \ref{thm:son}, so what remains to show is the necessity.

Because the system in \eqref{eq:son} is controllable on SO$(n)$, ${\rm Lie}(\mathcal{F})=\mathfrak{so}(n)$. By Theorem \ref{thm:son}, there is a subset $\mathcal{S}$ of $\mathcal{F}$ such that $|\mathcal{S}|=n-1$ and $\iota(\mathcal{S})=\s$ is an $n$-cycle 
in $(S_n,\cdot)$. This implies that the nontrivial orbit of $\iota(\mathcal{S})$ can only be $\{1,\dots,n\}$, and thus $\s\in[(1,\dots,n)]$. If $\mathcal{S}=\mathcal{F}$, then we are done. If not, i.e., $\S\subset\F$, then for any $\O_{ij}\in\mathcal{F}\backslash\mathcal{S}$, one can decompose $\s=(i,j)\cdot\eta$ for some permutation $\eta\in S_n$.
Therefore, $\tilde{\iota}(\O_{ij})*\tilde{\iota}(\mathcal{S})=[(i,j)]*[\s]=[(i,j)]*[(i,j)*\eta]=[(i,j)*(i,j)*\eta]=[(i,j)*\eta]=[(i,j)\cdot\eta]=[\s]$. Because $\s\in[(1,\dots,n)]$, $[\s]=[(1,\dots,n)]$ by the fact that every permutation in $S_n$ can only be in one equivalent class of $S_n/\sim$. Since $\O_{ij}\in\mathcal{F}\backslash\mathcal{S}$ was arbitrary, we conclude that $\tilde{\iota}(\mathcal{F})=\prod_{\O_{ij}\not\in\S}\tilde{\iota}(\O_{ij})*\prod_{\O_{ij}\in\S}\tilde{\iota}(\O_{ij})=\prod_{\O_{ij}\in\F\backslash\S}\tilde{\iota}(\O_{ij})*\tilde{\iota}(\mathcal{S})=[\s]=[(1,\dots,n)]$.
\end{proof}

Similar to Corollary \ref{cor:son}, we can identify the controllable submanifold of the system \eqref{eq:son} through the monoid $(S_n/\sim,*)$. 

\begin{corollary}
	\label{cor:*operation}
	The controllable submanifold of the system defined on {\rm SO}($n$) in \eqref{eq:son} is the integral manifold of the involutive distribution $\Delta=\Delta_1\oplus\cdots\oplus\Delta_k$, where $\Delta_l={\rm span}\{\O_{ij}X:i,j\in\mathcal{O}_l\}$, if and only if $\~{\iota}(\mathcal{F})=[\s_1\cdot\s_2\cdots\s_k]$, where $\s_l\in S_n$ are disjoint cycles and $\mathcal{O}_l$ denote the respective nontrivial orbits of $\s_l$ for $l=1,\dots,k$.
\end{corollary}
\begin{proof}
	(Necessity) By the assumption, we have $\Delta=\Delta_1\oplus\cdots\oplus\Delta_k$ with $\Delta_l={\rm span}\{\O_{ij}X:i,j\in\mathcal{O}_l\}$. By Corollary \ref{cor:son}, there exist subsets $\mathcal{F}_1,\dots,\mathcal{F}_k$ of $\mathcal{F}$ such that $\Lie(\mathcal{F}_l)={\rm span}\{\O_{ab}:a,b\in\mathcal{O}_l\}$ and $\iota(\mathcal{F}_l)=\s_l$ for $l=1,\dots,k$, where $\s_l$ is a cycle with the nontrivial orbit $\mathcal{O}_l$, and also $\Lie(\mathcal{F}_1\cup\cdots\cup\mathcal{F}_k)=\Lie(\mathcal{F}_1)\oplus\cdots\oplus\Lie(\mathcal{F}_k)=\Lie(\mathcal{F})$. Moreover, because $\mathcal{O}_l,l=1,\dots,k$ are disjoint, $\~{\iota}(\mathcal{F}_1\cup\cdots\cup\mathcal{F}_k)=[\s_1]*\cdots*[\s_k]=[\s_1*\cdots*\s_k]=[\s_1\cdots\s_k]$, which is the equivalent class of permutations with nontrivial orbits $\mathcal{O}_1,\dots,\mathcal{O}_k$. If $\mathcal{F}_1\cup\cdots\cup\mathcal{F}_k=\mathcal{F}$, then the proof is done. Otherwise, for any $\O_{ab}\in\mathcal{F}\backslash(\mathcal{F}_1\cup\cdots\cup\mathcal{F}_k)$, $\O_{ab}\in\Lie(\mathcal{F})=\Lie(\mathcal{F}_1)\oplus\cdots\oplus\Lie(\mathcal{F}_k)$, because $\mathcal{F}\subseteq\Lie(\mathcal{F})$. This implies that $\O_{ab}\in\Lie(\mathcal{F}_l)$ for some $l\in\{1,\dots,k\}$, and thus $\{a,b\}\subseteq\mathcal{O}_l$.
Consequently, by the commutativity and associativity of the binary operation $*$ on $S_n/\sim$, $[(a,b)]*\~{\iota}(\mathcal{F}_1\cup\cdots\cup\mathcal{F}_k)=[(a,b)]*[\s_1]*\cdots*[\s_k]=[\s_1]*\cdots*([(a,b)]*[\s_l])*\cdots*[\s_k]=[\s_1]*\cdots*([(a,b)*\s_l])*\cdots*[\s_k]=[\s_1]*\cdots*[\s_l]*\cdots*[\s_k]=\~{\iota}(\mathcal{F}_1\cup\cdots\cup\mathcal{F}_k)$. Because $\O_{ab}\in\mathcal{F}\backslash(\mathcal{F}_1\cup\cdots\cup\mathcal{F}_k)$ was arbitrary, we conclude $\~{\iota}(\mathcal{F})=\~{\iota}(\mathcal{F}\backslash(\mathcal{F}_1\cup\cdots\cup\mathcal{F}_k))*\~{\iota}(\mathcal{F}_1\cup\cdots\cup\mathcal{F}_k)=\~{\iota}(\mathcal{F}_1\cup\cdots\cup\mathcal{F}_k)=[\s_1\cdots\s_k]$.

(Sufficiency) Let $M\subseteq {\rm SO}(n)$ denote the controllable submanifold of the system in \eqref{eq:son}. Then, the LARC implies $T_IM=\Lie(\F)$. 
Together with Corollary \ref{cor:son} that any element $\O_{ab}\in\mathcal{B}$ such that $\{a,b\}\subseteq\mathcal{O}_i$ implies $\O_{ab}\in\Lie(\F)$, we have $\Delta=\Delta_1\oplus\cdots\oplus\Delta_k\subseteq{\rm Lie}(\mathcal{F})$. To complete the proof, we 
will show $\Lie(\mathcal{F})\subseteq\Delta$, and equivalently, we will prove that if $\O_{ab}\in\mathcal{B}$ with $\{a,b\}\nsubseteq\mathcal{O}_i$ for any $i=1,\dots,k$, i.e., $\O_{ab}\not\in\Delta$, then $\O_{ab}\notin\text{Lie}(\F)$. Because the index set of $\mathcal{F}$ is $\{i_1,j_1\dots,i_m,j_m\}=\mathcal{O}_1\cup\cdots\cup\mathcal{O}_k$, $k\leq m$, if $a\not\in\mathcal{O}_1\cup\cdots\cup\mathcal{O}_k$ or $b\not\in\mathcal{O}_1\cup\cdots\cup\mathcal{O}_k$, then $\O_{ab}\not\in\Lie(\F)$. 
Hence, what remains to show is the case in which $a\in\mathcal{O}_i$ and $b\in\mathcal{O}_j$, for some $i\neq j$, does not occur. 
Suppose that there exists $\O_{ab}\in\Lie(\mathcal{F})$ with $a\in\mathcal{O}_i$ and $b\in\mathcal{O}_j$ for some $i\neq j$, then we have $[(a,b)]*\~{\iota}(\mathcal{F})=\~{\iota}(\mathcal{F})$ by Remark \ref{rmk:redundant}. However, because $\{a,b\}\not\subseteq\mathcal{O}_i$ and $\{a,b\}\not\subseteq\mathcal{O}_j$, $[(a,b)]*[\s_i]*[\s_j]=[\s_i]*[(a,b)]*[\s_j]=[\s_i*(a,b)*\s_j]=[\s_i\cdot(a,b)\cdot\s_j]$ is the equivalent class of cycles with nontrivial orbit $\mathcal{O}_i\cup\mathcal{O}_j$ due to the fact that $a\in\mathcal{O}_i$ and $b\in\mathcal{O}_j$. But $\mathcal{O}_i\cup\mathcal{O}_j$ cannot be a nontrivial orbit of $\~{\iota}(\mathcal{F})$ by the disjointedness assumption of $\mathcal{O}_1,\dots,\mathcal{O}_k$, and thus $[(a,b)]*\~{\iota}(\mathcal{F})\neq\~{\iota}(\mathcal{F})$ which is a contradiction. 
\end{proof}

\subsection{Systems with Drift and Governed by Nonstandard Basis Vector Fields}
\label{sec:drift}
It is straightforward to realize that the method of examining controllability over the monoid $(S_n/\sim,*)$ also works for systems defined on SO$(n)$ with drift of the form
\begin{align}
	\label{eq:standard_drift}
	\dot{X}(t)=\left[\O_{i_0j_0}+\sum_{k=1}^ku_k(t)\O_{i_kj_k}\right]X, \quad X(0)=I,
\end{align}
because SO$(n)$ is a compact manifold so that the LARC is applicable \cite{Elliott09}. Specifically, the system in \eqref{eq:standard_drift} is controllable on SO$(n)$ if and only if $\~{\iota}(\F)=[(1,\dots,n)]$, where $\F=\{\O_{i_0j_0},\dots,\O_{i_mj_m}\}\subseteq\B$. In addition, if the system is not completely controllable, its controllable submanifold is the integral manifold of the involutive distribution $\Delta=\Delta_1\oplus\cdots\oplus\Delta_k$ if and only if $\~{\iota}(\mathcal{F})=[\s_1\cdot\s_2\cdots\s_k]$, where $\Delta_l={\rm span}\{\O_{ij}X:i,j\in\mathcal{O}_l\}$, $\s_l\in S_n$ are disjoint cycles, and $\mathcal{O}_l$ denote the nontrivial orbit of $\s_l$ for $l=1,\dots,k$.

Moreover, we would like to comment on the case in which the vector fields of the system in \eqref{eq:standard_drift} consist of nonstandard bases of $\mathfrak{so}(n)$. A general model system is of the form
\begin{eqnarray}
	\label{eq:nonstandard}
	\dot{X}(t)=B_0X(t)+\sum_{k=1}^mu_kB_kX(t), \quad X(0)=I,
\end{eqnarray}
where $X(t)\in\text{SO}(n)$ and $B_i\in\mathfrak{so}(n)$ for $i=0,1,\ldots,m$. The methodology and results about analyzing controllability over the monoid $(S_n/\sim,*)$ in terms of permutation cycles presented above in Section \ref{sec:*} remain applicable to the system in \eqref{eq:nonstandard}, if $\Lie(\{B_0,\dots,B_m\})\not\subseteq\mathfrak{so}(4)\oplus\cdots\oplus\mathfrak{so}(4)$, namely, the Lie algebra generated by its drift and control vector fields is not a Lie subalgebra of a direct sum of $\mathfrak{so}(4)$. This conjecture remains to be proved rigorously and is beyond the scope of this work. It is left for future investigations. 

To shed light on this observation, let's consider the system defined on SO(4), given by
\begin{align}
	\label{eq:non_standard_basis}
	\dot{X}(t)=\big[u_1(\O_{12}+\O_{34})+u_2\O_{23}\big]X(t), \quad X(0)=I,
\end{align}
where the set of control vector fields is $\mathcal{G}=\{\O_{12}+\O_{34},\O_{23}\}$. It is easy to verify by the LARC that the system in \eqref{eq:non_standard_basis} is not controllable on SO(4), because $\Lie(\mathcal{G})={\rm span}\{\O_{12}+\O_{34},\O_{13}-\O_{24},\O_{14}+\O_{23},\O_{23}\}={\rm span}\{\O_{12}+\O_{34},\O_{13}-\O_{24},\O_{14}+\O_{23}\}\oplus{\rm span}\{\O_{23}\}$ is a proper Lie subalgebra of $\mathfrak{so}(4)$ isomorphic to $\mathfrak{so}(3)\oplus\mathfrak{so}(2)$. However, if in this case we take $\F=\{\O_{12},\O_{34},\O_{23}\}$, then the resulting equivalence class of permutations is $\~\iota(\F)=[(1,2)]*[(3,4)]*[(2,3)]=[(1,2,3,4)]$, which, by Theorem \ref{thm:*operation}, implies controllability of the system, yielding a contradiction. The failure of the application of Theorem \ref{thm:*operation} in this case is caused by the fact that the Lie algebra $\mathfrak{so}(4)$, which is isomorphic to $\mathfrak{so}(3)\oplus\mathfrak{so}(3)$, is not simple.

However, when an additional control input is available for the system in \eqref{eq:non_standard_basis}, for instance, 
\begin{align*}
	\dot{X}(t)=\big[u_1(\O_{12}+\O_{34})+u_2\O_{23}+u_3\O_{12}\big]X(t), \ X(0)=I,
\end{align*}
then the system is controllable by the LARC, since $\Lie(\mathcal{G}')=\mathfrak{so}(4)$, where $\mathcal{G}'=\{\O_{12}+\O_{34},\O_{23},\O_{12}\}$. From the symmetric group point of view, the permutation elements resulting from $\iota(\mathcal{G})$ and $\iota(\mathcal{G}')$ are $\overline{\mathcal{G}}=\{(1,2)(3,4),(2,3)\}$ and $\overline{\mathcal{G}'}=\{(1,2)(3,4),(2,3),(1,2)\}$, respectively, if we extend the domain of $\iota$ from $\mathcal{P}(\B)$ to $\mathcal{P}(\mathfrak{so}(4))$ and define $\iota(\O_{12}+\O_{34})=(1,2)(3,4)$. Then, it can be computed that the group generated by $\overline{\mathcal{G}}$ is $D_8$, the dihedral group of order 8, which is a proper subgroup of $S_4$; however, $\overline{\mathcal{G'}}$ generates the entire group $S_4$. Motivated by this observation, we assert that the system in \eqref{eq:nonstandard} is controllable on SO$(n)$ if the subgroup generated  by $\{\s_0,\dots,\s_m\}$ is $S_n$, where $\s_k=\iota(B_k)$, $k=0,1,\dots,m$.

\begin{remark}[\textbf{Bilinear Systems Induced by Group Actions}] 
	\rm
	Another perspective to analyze systems defined on Lie groups is through the study of their group actions. For example, the action of the system on SO$(n)$ in \eqref{eq:son} on $\mathbb{R}^n$ is of the form
	\begin{align}
		\label{eq:son_action}
		\dot{x}(t)=\O_{i_0j_0}x(t)+\sum_{k=1}^mu_k\O_{i_kj_k}x(t), \quad x(0)=x_0,
	\end{align}
	where $x(t)\in\mathbb{R}^n$ for all $t\geq0$. As we know, the orbits of the Lie group action of SO$(n)$ on $\mathbb{R}^n$ are the origin and the spheres centered at the origin, and therefore the state space of the system in \eqref{eq:son_action} is $\mathbb{S}^{n-1}_{\|x_0\|}=\{z\in\mathbb{R}^n:\|z\|=\|x_0\|, x_0\neq 0\}$, the $n-1$ dimensional sphere in $\mathbb{R}^n$ centered at the origin with radius $\|x_0\|$. The framework of examining controllability over the monoid $(S/\sim,*)$ is also applicable to this bilinear system in $\mathbb{R}^n$. 
	Let $e_k$ denote the $k^{\rm th}$ standard basis of $\mathbb{R}^n$, then we have $\O_{ij}e_k=\delta_{ik}e_i-\delta_{jk}e_j$, or equivalently, 
$$\O_{ij}e_k=\begin{cases} e_j, & i=k, \\ -e_i, & j=k, \\ 0, & \text{otherwise}. \end{cases}$$
This relation illustrates that $\O_{ij}$ serves as a bridge transferring the state between $e_i$ and $e_j$. Following this idea, the system in \eqref{eq:son_action} is controllable on $\mathbb{S}^{n-1}_{\|x_0\|}$ if and only if $\~\iota(\F)=[(1,\dots,n)]$, i.e., $\Lie(\F)$ contains bridges between any two standard bases of $\mathbb{R}^n$ so that any state $x(t)$ in $\mathbb{R}^n$ can be reached, where $\F=\{\O_{i_0j_0},\dots,\O_{i_mj_m}\}$. 
\end{remark}

\section{Bilinear Systems Defined on Undirected Graphs}
\label{sec:graph}
The new notion of mapping Lie brackets to permutations developed in Sections \ref{sec:SO(n)_controllability} and \ref{sec:controllability_monoid} can be directly applied to analyze controllability of broader classes of control systems, ncluding multi-agent systems, symmetric Markov chains, and quantum networks \cite{Zhang15}, and can be generalized to examine systems defined on compact, connected Lie groups with semisimple Lie algebras, such as SU$(n)$.

\subsection{Formation Control of Multi-agent Systems}
\label{sec:multiagent}
A generic question about formation control is concerned with how interconnected agents communicate and cooperate in a centralized or decentralized fashion towards a consensus 
\cite{Helmke13, Belabbas13}. The formation and, moreover, path controllability of multi-agent systems are conventionally analyzed based on the LARC \cite{Chen14}. Here, we study formation controllability by analyzing it over symmetric groups.


Specifically, we consider the motion of a network of $N$ agents in $\mathbb{R}^n$. Let $\Gamma=(V,E)$ be an undirected graph associated with this multi-agent system with the set of vertices $V=\{1,\dots,N\}$ and the set of edges $E$. Each agent follows the dynamic law  \cite{Chen14},
\begin{align}
	\label{eq:multi_agent_ith_agent}
	\dot{x}_i=\sum_{j\in V(i)}u_{ij}(x_j-x_i),\quad i=1,\dots,N,
\end{align}
where $x_i(t)\in\mathbb{R}^n$, $V(i)=\{j\in V:(i,j)\in E\}$ is the set of vertices that are adjacent to agent $i$, and $u_{ij}=u_{ji}$ are piecewise constant controls determining the reciprocal interaction between $x_i$ and $x_j$. The nondegenrate configuration space (state space) of this system, denoted by 
$\mathbb{P}$, contains points $(x_1,\dots,x_N)\in(\mathbb{R}^n)^N$ such that $\sum_{i=1}^Nx_i=0$, $x_i\neq x_j$ for all $i\neq j$, and ${\rm span}\{x_1-x_2,\dots,x_1-x_N\}=\mathbb{R}^n$ \cite{Chen14}.

Let $e_i$ denote the $i^{th}$ standard basis of $\mathbb{R}^N$, $i=1,\dots,N$, and define the matrix $A_{ij}\in\mathbb{R}^{N\times N}$ by $A_{ij}=e_ie_j'+e_je_i'-e_ie_i'-e_je_j'$, which are symmetric with zero row and column sum. In addition, let $X\in\mathbb{R}^{N\times n}$ be a matrix whose $i^{th}$ row defines the state vector of agent $i$, i.e., 
$$X=\left[\begin{array}{c} x_1' \\ \hline\vdots \\ \hline x_N' \end{array}\right],$$
then, by \eqref{eq:multi_agent_ith_agent}, the system of $N$ agents follows the dynamic equation
\begin{align}
	\label{eq:multi_agent}
	\dot{X}=\sum_{(i,j)\in E}u_{ij}A_{ij}X, \quad X(0)=X_0.
\end{align}
Moreover, for any $i,j,k=1,\dots,N$, we define an antisymmetric matrix $B_{ijk}\in\mathbb{R}^{N\times N}$ by $B_{ijk}=(e_ie_k'-e_ke_i')-(e_ie_j'-e_je_i')-(e_je_k'-e_ke_j')$, and thus each $B_{ijk}$ has zero row and column sum as well. Define the sets $\mathbb{A}=\{A_{ij}:i,j=1,\dots,N\}$, $\mathbb{B}=\{B_{ijk}:i,j,k=1,\dots,N\}$, and the real vector spaces $\mathfrak{g}_1={\rm span}(\mathbb{A})$ and $\mathfrak{g}_2={\rm span}(\mathbb{B})$, then the direct sum
\begin{align}
  \label{eq:g}
  \mathfrak{g}=\mathfrak{g}_1\oplus\mathfrak{g}_2
\end{align}
defines the set of matrices whose rows and columns sum to zero.

\begin{lemma}
	$\mathfrak{g}$ is an $(N-1)^2$ dimensional Lie algebra.
\end{lemma}
\begin{proof}
Since the elements in $\mathbb{A}$ are linearly independent, they form a basis of $\mathfrak{g}_1$. Therefore, the dimension of $\mathfrak{g}_1$ is $\frac{1}{2}N(N-1)$, which is the cardinality of $\mathbb{A}$. Similarly, the elements in the set $\{B_{1jk}:j,k=1,\dots,N\}$ form a basis of $\mathfrak{g}_2$ with dimension $(N-1)(N-2)/2$. As a direct sum of two vector spaces, $\mathfrak{g}$ is also a vector space and its dimension is the sum of the dimension of $\mathfrak{g}_1$ and $\mathfrak{g}_2$, that is, $\dim(\mathfrak{g})=N(N-1)/2+(N-1)(N-2)/2=(N-1)^2$.

Moreover, the symmetry and antisymmetry of $\mathfrak{g}_1$ and $\mathfrak{g}_2$, respectively, 
yield the relations $[\mathfrak{g}_1,\mathfrak{g}_1]\subseteq\mathfrak{g}_2$, $[\mathfrak{g}_2,\mathfrak{g}_2]\subseteq\mathfrak{g}_2$, and $[\mathfrak{g}_1,\mathfrak{g}_2]\subseteq\mathfrak{g}_1$, where $[\cdot,\cdot]$ denotes the Lie bracket. Because any elements $M,N\in\mathfrak{g}$ can be decomposed into $M=M_1+M_2$ and $N=N_1+N_2$ with $M_1,N_1\in\mathfrak{g}_1$ and $M_2,N_2\in\mathfrak{g}_2$, we obtain $[M,N]=[M_1+M_2,N_1+N_2]=[M_1,N_1]+[M_1,N_2]+[M_2,N_1]+[M_2,N_2]\in\mathfrak{g}$, since $[M_1,N_1],[M_2,N_2]\in\mathfrak{g}_2$ and $[M_1,N_2],[M_2,N_1]\in\mathfrak{g}_1$. Therefore, $\mathfrak{g}$ is closed under the Lie bracket operation and hence is a Lie algebra, because $M$ and $N$ are arbitrary.
\end{proof}

Specifically, we have the following Lie bracket relations:
\begin{align}
	\label{eq:Lie_Aij_1}
	&[A_{ij},A_{jk}]=[A_{jk},A_{ik}]=[A_{ik},A_{ij}]=B_{ijk}, \\
	\label{eq:Lie_Aij_2}
	&[A_{ij},B_{ijk}]=2(A_{jk}-A_{ik}), \\
	&[A_{jk},B_{ijk}]=2(A_{ik}-A_{ij}), \nonumber \\
	&[A_{ik},B_{ijk}]=2(A_{ij}-A_{jk}) \nonumber,	
\end{align}
and $[A_{ij},A_{kl}]\neq0$ when $i=k$, $i=l$, $j=k$, or $j=l$. By \eqref{eq:Lie_Aij_1}, it is evident that $B_{ijk}\in\Lie(\{A_{ij},A_{jk}\})$. In addition, using \eqref{eq:Lie_Aij_2}, together with \eqref{eq:Lie_Aij_1}, gives $A_{ik}=(2A_{jk}-[A_{ij},[A_{ij},A_{jk}]])/2$, which implies that $A_{ik}\in\Lie(\{A_{ij},A_{jk}\})$. Moreover, $B_{\s(i)\s(j)\s(k)}=(-1)^{\s}B_{ijk}$ and thus $B_{\s(i)\s(j)\s(k)}\in\Lie(\{A_{ij},A_{jk}\})$ for any $\s\in S_3$, a permutation on three letters. Given this with \eqref{eq:Lie_Aij_1} and \eqref{eq:Lie_Aij_2}, we observe that all iterated Lie brackets of $\{A_{ij},A_{jk}\}$ are linear combinations of $\{A_{ij},A_{jk}, A_{ik},B_{ijk}\}$, which implies $\Lie(\{A_{ij},A_{jk}\})={\rm span}\{A_{ij},A_{jk}, A_{ik},B_{ijk}\}$.

\begin{theorem}
	\label{thm:multiagent}
The multi-agent system in \eqref{eq:multi_agent} is controllable on $\mathbb{P}$ 
  if and only if $\~{\iota}(\F)=[(1,\dots,N)]$, where $\F=\{A_{ij}:(i,j)\in E\}$. 
\end{theorem}
\begin{proof}
Observe that all iterated Lie brackets of any two vector fields in $\F$, e.g., $\{A_{ij},A_{jk}\}$, are linear combinations of some elements in $\{A_{ij},A_{jk}, A_{ik},B_{ijk}\}$, which implies $\Lie(\{A_{ij},A_{jk}\})={\rm span}\{A_{ij},A_{jk}, A_{ik},B_{ijk}\}$. In addition, by the definition of $\tilde\iota$, we have $\tilde{\iota}(\{A_{ij},A_{jk}\})=[(i,j)*(j,k)]=[(i,j,k)]$. Then, following a similar proof of Theorem \ref{thm:*operation}, it can be shown that $\tilde{\iota}(\F)=[(1,\dots,N)]$ if and only if $A_{ij}\in\Lie(\F)$ for all $i,j\in\{1,\dots,N\}$. In this case, the property $\Lie(\{A_{ij},A_{jk}\})={\rm span}\{A_{ij},A_{jk}, A_{ik},B_{ijk}\}$ further implies $B_{ijk}\in\Lie(\F)$ for all $i,j,k\in\{1,\dots,N\}$, and hence $\Lie(\F)={\rm span}\{A_{ij},B_{ijk}:1\leq i,j,k\leq N\}=\mathfrak{g}$, the Lie algebra defined in \eqref{eq:g}. This concludes by the LARC that the system in \eqref{eq:multi_agent} is controllable if and only if $\tilde{\iota}(\F)=[(1,\dots,N)]$.
\end{proof}

\begin{corollary}
	\label{cor:multiagent}
	The controllable submanifold of the system in \eqref{eq:multi_agent} is the integral manifold of the involutive distribution $\Delta=\Delta_1\oplus\cdots\oplus\Delta_l$ containing $X_0$, where $\Delta_r={\rm span}\{A_{ij}X,B_{ijk}X:i,j,k\in\mathcal{O}_r\}$, if and only if $\~\iota(\mathcal{F})=[\s_1\cdots\s_l]$, where $\s_l\in S_n$ are disjoint cycles and $\mathcal{O}_r$ denotes the largest orbit of $\s_r$ for $r=1,\dots,l$.
\end{corollary}
\begin{proof}
	The proof follows that of Corollary \ref{cor:*operation}.
\end{proof}

\subsection{Control of Symmetric Markov Chains}
\label{sec:markov}
A natural class of stochastic systems defined on graphs are Markov chains, which have broad applications from web search and gene expressions to 
process control. Continuous-time finite-state (CTFS) Markov chains are widely used to model stochastic processes in these areas. Let's consider a CTFS
symmetric Markov chain $X(t)$ on the finite set $\mathbb{S}=\{1,\dots,n\}$, which is reversible with respect to the uniform probability measure $\pi_i=1/n$ for all $i\in \mathbb{S}$ \cite{Lawler06}. Let $p_k(t)$ be the probability of the chain at state $k$, i.e., $p_k(t)={\rm Prob}(X(t)=k)$ for $k=1,\ldots,n$. Then, the dynamics of $P(t)=(p_1(t),\dots,p_n(t))'$ follow
\begin{align}
	\label{eq:p_evolution}
	\dot{P}=MP, \quad P(0)=P_0,
\end{align}
where $P_0$ is the probability distribution of $X(0)$ and $M\in\mathbb{R}^{n\times n}$ is the intensity matrix, which is symmetric with zero sum rows and columns. 
Because this Markov chain is reversible with respect to $\pi$, we have
\begin{align}
\label{eq:invariant}
\sum_{j\neq i}M_{ij}\pi_j=\sum_{i\neq j}M_{ji}\pi_i=M_i\pi_i,
\end{align}
where $M_{ij}$, the $ij^{th}$ entry of $M$, denotes the transition rate from state $i$ to state $j$, and $M_i=\sum_{j\neq i}M_{ji}$ denotes the total rate at which the chain is changing from state $i$. In addition, \eqref{eq:invariant} also implies that $\pi$ is an invariant measure for $X(t)$. Therefore, with constant transition rate, the probability distribution $P(t)$ of the chain will eventually converge to the uniform distribution.  

Now, suppose that one can manipulate the rates $M_{ij}$, then the controlled dynamics of $P(t)$ in \eqref{eq:p_evolution} can be expressed as 
\begin{align}
	\label{eq:p_control}
	\dot{P}=\sum_{\{i,j: M_{ij}\neq0\}}u_{ij}A_{ij}P, \quad P(0)=P_0,
\end{align}
where $A_{ij}=A_{ji}=e_ie_j'+e_je_i'-e_ie_i'-e_je_j'\in\mathbb{A}$. 
Notice that the system in \eqref{eq:p_control} is of the same form as that 
in \eqref{eq:multi_agent}, and, in particular, 
it represents the dynamics of an agent in the multi-agent system \eqref{eq:multi_agent}. Consequently, following Theorem \ref{thm:multiagent}, the system in \eqref{eq:p_control} is controllable on the $(n-1)$-simplex $\Delta^{n-1}=\{(p_1,\dots,p_n):\sum_{i=1}^np_i=1\}$ if and only if $\tilde\iota(\F)=[(1,\dots,n)]$, where $\F=\{A_{ij}:1\leq i,j\leq n,M_{ij}\neq0\}$ is the set of control vector fields. In this case, for any $i\in\{1,\dots,n\}$, there exists some $j\in\{1,\dots,n\}$, such that $A_{ij}\in\F$, or equivalently $M_{ij}\neq0$, which implies that the Markov chain $X(t)$ is irreducible.

However, if the Markov chain $X(t)$ is reducible and has $l$ nontrivial communication classes $\mathcal{O}_1,\dots,\mathcal{O}_l$, then we have $M_{ij}\neq0$ only if $\{i,j\}\subseteq\mathcal{O}_r$ for some $r=1,\dots,l$, where $l<n$. In this case, we obtain $\~{\iota}(\F)=[\s_1\cdots\s_l]$ such that $\s_r$ is a cycle with the nontrivial orbit $\mathcal{O}_r$ for each $r=1,\dots,l$. According to Corollary \ref{cor:multiagent}, the controllable submanifold of the system in \eqref{eq:p_control} is the integral manifold of the involutive distribution $\Delta=\Delta_1\oplus\cdots\oplus\Delta_l$ containing $P_0$, where $\Delta_r={\rm span}\{A_{ij}P,B_{ijk}P:i,j,k\in\mathcal{O}_r\}$. We also know that the set of the feasible probability distributions of the Markov chain $X(t)$ in \eqref{eq:p_control} is $D=\{(p_1,\dots,p_n)'\in \Delta^{n-1}:\sum_{i\in\mathcal{O}_r}p_{i}=\sum_{i\in\mathcal{O}_r}p_{i}(0){\ \rm for\ }r=1,\dots,l, {\ \rm and\ } p_j=p_j(0){\ \rm for\ }j\in\mathbb{S}\backslash(\bigcup_{r=1}^l\mathcal{O}_r)\}$. Hence, Corollary \ref{cor:multiagent} implies that $T_PD=\Delta_P$ for every point $P\in D$, where $D$ is a submanifold of $\Delta^{n-1}$ 
and $\Delta_P$ is the distribution evaluated at $P$.

\section{Interpretation and Visualization of Controllability over Graphs}
\label{sec:visualization}
Inspired by the idea of mapping controllability analysis to permutation compositions, 
this fundamental property can be interpreted and visualized by graphs. The system on SO$(n)$ in \eqref{eq:son}, the network of multiple agents in \eqref{eq:multi_agent}, and the Markov chain on a simplex in \eqref{eq:p_control} are control systems of the form
\begin{align}
	\label{eq:bilinear}
	\frac{d}{dt}X(t)=\sum_{k=1}^mu_i(t)H_{i_kj_k}X(t), \quad X(0)=X_0,
\end{align}
where $H_{i_kj_k}\in\mathcal{M}$ and $\mathcal{M}=\mathcal{B}$ or $\mathbb{A}$. 
Associated with a system in \eqref{eq:bilinear}, one can define an undirected, unweighted graph $\Gamma=(V,E)$ according to the control vector fields, $\F=\{H_{i_1j_1},\dots,H_{i_mj_m}\}$, where $V=\{1,\dots,n\}$ 
and $E=\{(i,j):H_{ij}\in\F\}$ 
representing the indices of $\F$. Notice that for the multi-agent system and the symmetric Markov chain, $\Gamma$ is the graph describing the interactions of the agents and the transitions between the states in respective cases. 


Recall that the system in \eqref{eq:bilinear} is controllable if $\~\iota(\F)=[(1,\dots,n)]$. In this case, for any $H_{ij}$, there exist $\pm H_{ik}$ and $\pm H_{kj}$ in $\F$ such that $H_{ij}$ can be generated by iterated Lie brackets of them. Translating the same idea to the graph $\Gamma=(V,E)$, this is equivalent to saying that for any $i,j\in V$, there exists $k\in V$ such that $(i,k), (j,k)\in E$, and hence the nodes $i,j,k$ are connected by a path {$i-k-j$. Then $\~\iota(\F)=[(1,\dots,n)]$ implies that $\Gamma=(V,E)$ is a connected graph. Moreover, we have also known that if $\~{\iota}(\mathcal{F})=[\s_1\cdot\s_2\cdots\s_l]$ for some disjoint cycles $\s_1,\s_2\cdots,\s_l$ with nontrivial orbits $\mathcal{O}_1,\dots,\mathcal{O}_l$, respectively, then the controllable submanifold of the system is the integral manifold of the involutive distribution $\Delta=\Delta_1\oplus\cdots\oplus\Delta_l$, where $\Delta_r={\rm span}\{\O_{ij}X:i,j\in\mathcal{O}_r\}$ if $\mathcal{M}=\B$ and $\Delta_r={\rm span}\{A_{ij}X, B_{ijk}X:i,j,k\in\mathcal{O}_r\}$ if $\mathcal{M}=\mathbb{A}$ for each $r=1,\ldots,l$. In this case, the graph $\Gamma$ has $l$ nontrivial connected components $\Gamma_1=(V_1,E_1),\dots,\Gamma_l=(V_l,E_l)$ with $V_r=\mathcal{O}_r$ for $r=1,\dots,l$. These observations illuminate the relationship between controllability of the bilinear system in \eqref{eq:bilinear} and the connectivity of the associated graph $\Gamma$, which are summarized in the following theorem and corollary.

\begin{theorem}
\label{thm:graph}
A bilinear system in the form of \eqref{eq:bilinear} is controllable if and only if its associated graph $\Gamma=(V,E)$ is connected, where $V=\{1,\dots,n\}$ and $E=\{(i_1,j_1),\dots,(i_m,j_m)\}$.
\end{theorem}

\begin{corollary}
\label{cor:graph}
Let $\Gamma=(V,E)$ be the graph associated to a bilinear system in the form of \eqref{eq:bilinear}. Then, the graph $\Gamma=(V,E)$ has $l$ nontrivial connected components $\Gamma_1=(V_1,E_1)$, $\dots$, $\Gamma_l=(V_l,E_l)$ if and only if  the controllable submanifold of the system is the integral manifold of the involutive distribution $\Delta=\Delta_1\oplus\cdots\oplus\Delta_l$ containing $X_0$, where $\Delta_r={\rm span}\{\O_{ij}X:i,j\in V_r\}$ if $\mathcal{M}=\B$ and $\Delta_r={\rm span}\{A_{ij}X,B_{ijk}X:i,j,k\in V_r\}$ if $\mathcal{M}=\mathbb{A}$ for each $r=1,\dots,l$.
\end{corollary}


\begin{remark}[\textbf{Computational Complexity for Evaluating Controllability}]
\rm The application of LARC for analyzing controllability involves computing iterated Lie brackets and examining the linear independence of the resulting vector fields, e.g., by Gaussian elimination. The complexity of such computations is on the order of $O(n^3)$. In this newly developed framework, controllability of bilinear systems, governed by the vector fields that form a magma structure, can be analyzed in terms of the length of permutation orbits or connectivity of the graph associated with the system. Computationally, permutation multiplications are operated by multiplications of permutation matrices. These are sparse matrices, and the computational complexity of their multiplications can be reduced down to the order of $O(n^{2.73})$, e.g., by the Coppersmith-Winograd algorithm \cite{Coppersmith90}. In addition, the complexity of the algorithms for computing graph connectivity and identifying connected components is on the order of $O(n^2)$ (the worse case scenario) \cite{Tarjan72, Iverson15}. 
\end{remark}


\section{conclusion}
In this paper, we introduced a new algebraic approach for the characterization of controllability and identification of controllable submanifold for 
time-invariant bilinear systems broadly defined on compact connected Lie groups and on undirected graphs, such as multi-agent systems or Markov chains. The key innovation was to establish a map from the Lie bracket operations over the control vector fields to permutation compositions on a symmetric group, so that the classical LARC can be translated to a condition in terms of permutation cycles. This new development further enabled a graph representation of controllability that involves interpreting controllability by the graph connectivity and characterizing controllable submanifold by nontrivial connected subgraphs. The established framework offered an alternative path, other than the LARC, to understand controllability and provided insight into the design of effective computational methods for computing controllability and identifying controllable submanifold through permutation operations and graph traversal algorithms.

\section{Appendix}

\begin{lemma}
	\label{lem:*orbit}
	For any cycles $\s,\eta\in S_n$ such that $|\mathcal{O}_{\s}\cap\mathcal{O}_{\eta}|\geq 1$, where $\mathcal{O}_{\s}$ and $\mathcal{O}_{\eta}$ denote the nontrivial orbits of $\s$ and $\eta$, respectively, then $\s*\eta\in[\pi]$ where $\pi$ is a cycle with nontrivial orbit $\mathcal{O}_{\s}\cup\mathcal{O}_{\eta}$. Consequently, the commutativity relation $\s*\eta\sim \eta*\s$ holds.
\end{lemma}

\begin{proof}
Recall in \eqref{eq:*operation_r} and \eqref{eq:*operation_l} that different calculations of $\s*\eta$ involve the same number of the `$\cdot$' operations, and thus they will result in different permutations with the same order in $S_n$, where the order of an element in a group is defined to be the cardinality of the cyclic subgroup generated by this element. Applying the calculation in \eqref{eq:*operation_r} yields that $\pi=\s*\eta$ is a cycle with the nontrivial orbit $\mathcal{O}_{\s}\cup\mathcal{O}_{\eta}$ (can be shown in detail by induction on the length of $\s$). Let $k=|\mathcal{O}_{\s}\cup\mathcal{O}_{\eta}|$ be the order of $\pi$, and $\pi'$ be a permutation obtained by using \eqref{eq:*operation_l}, 
then $\pi'$ is also of order $k$ and has a nontrivial orbit containing $\mathcal{O}_{\s}\cup\mathcal{O}_{\eta}$, which implies $\pi'$ must be a $k$-cycle with the nontrivial orbit  $\mathcal{O}_{\s}\cup\mathcal{O}_{\eta}$. This concludes $\s*\eta\in[\pi]$.
\end{proof}

Lemma \ref{lem:*orbit} indicates that the $*$ operation avoids the degeneracy case described in Remark \ref{rem:degeneracy} when operating on cycles with overlapping elements. On the other hand, for disjoint cycles $\s,\eta\in S_n$, the $*$ operation is reduced to the `$\cdot$' operation, and thus $\s*\eta\in[\s\cdot\eta]$. Because every permutation in $S_n$ is a product of disjoint cycles, the above observation immediately leads to the following consequences.

\begin{corollary}
	\label{cor:*orbit}
	For any permutations $\s,\eta,\s',\eta' \in S_n$, if $\s\sim\s'$ and $\eta\sim\eta'$, then $\s*\eta\sim\s'*\eta'$.
\end{corollary}	
\begin{proof} 
	The proof follows from Lemma \ref{lem:*orbit} and the fact that every permutation in $S_n$ is a product of disjoint cycles.
\end{proof}

\begin{corollary}
	\label{cor:*commutativity}
	For any permutations $\s,\eta\in S_n$, $\s*\eta\sim\eta*\s$.
\end{corollary}
\begin{proof} 
	By Lemma \ref{lem:*orbit}, the commutativity of $*$ holds for cycles. Because every permutation in $S_n$ is a product of disjoint cycles that are commutable, the result follows.
\end{proof}

Corollary \ref{cor:*commutativity} implies that $*$ is commutative on $S_n/\sim$. Fortunately, not only the commutativity of $*$ holds, but also the associativity as shown in the following lemma.

\begin{lemma}
	\label{lem:*associativity}
	For any permutations $\s,\eta,\xi\in S_n$, $(\s*\eta)*\xi\sim\s*(\eta*\xi)$. 
\end{lemma}
\begin{proof}
	Because every permutation can be decomposed as a product of disjoint cycles under the binary operation $*$, without loss of generality, we only consider the case for cycles. 
 
	If $\s,\eta$ and $\xi$ are pairwise disjoint cycles, then the $*$ operation is reduced to the group operation `$\cdot$' on $S_n$,  i.e., $(\s*\eta)*\xi=(\s\cdot\eta)\cdot\xi$ and $\s*(\eta*\xi)=\s\cdot(\eta\cdot\xi)$. Because `$\cdot$', as a group operation, is associative, i.e., $(\s\cdot\eta)\cdot\xi=\s\cdot(\eta\cdot\xi)$, then we have $(\s*\eta)*\xi=(\s\cdot\eta)\cdot\xi=\s\cdot(\eta\cdot\xi)=\s*(\eta*\xi)$, which implies $(\s*\eta)*\xi\sim\s*(\eta*\xi)$. 

	On the other hand, if $\s,\eta$ and $\xi$ are not pairwise disjoint, let $\mathcal{O}_{\s},\mathcal{O}_{\eta}$ and $\mathcal{O}_{\xi}$ denote their nontrivial orbits, respectively, then there are two possibilities: (i) one of the three sets $\mathcal{O}_{\s},\mathcal{O}_{\eta}$ and $\mathcal{O}_{\xi}$ has nonempty intersections with both of the other two sets, then according to Lemma \ref{lem:*orbit}, both of $(\s*\eta)*\xi$ and $\s*(\eta*\xi)$ are cycles with the nontrivial orbit $\mathcal{O}_{\s}\cup\mathcal{O}_{\eta}\cup\mathcal{O}_{\xi}$; and (ii) one of $\s,\eta$ and $\xi$ is disjoint with the other two, without loss of generality, assuming that $\eta$ is disjoint with $\s$ and $\xi$. Therefore, $\s*\eta=\s\cdot\eta=\eta\cdot\s=\eta*\s$ is a permutation composed of two disjoint cycles whose nontrivial orbits are $\mathcal{O}_{\s}$ and $\mathcal{O}_{\eta}$, respectively. Hence, we have $(\s*\eta)*\xi=\eta\cdot\s*\xi=\eta\cdot(\s*\xi)$. According to Lemma \ref{lem:*orbit}, $\s*\xi$ is a cycle with the nontrivial orbit $\mathcal{O}_{\s}\cup\mathcal{O}_{\xi}$. Since $\eta$ is disjoint from $\s$ and $\xi$, $\eta\cdot(\s*\xi)$ is a permutation composed of two disjoint cycles with the respective nontrivial orbits $\mathcal{O}_{\eta}$ and $\mathcal{O}_{\s}\cup\mathcal{O}_{\xi}$. Similarly, $\s*(\eta*\xi)=\s*(\eta\cdot\xi)=\s*(\xi\cdot\eta)=(\s*\xi)\cdot\eta$ is also a permutation as a product of two disjoint cycles whose nontrivial orbits are $\mathcal{O}_{\eta}$ and $\mathcal{O}_{\s}\cup\mathcal{O}_{\xi}$, respectively. Therefore,  $(\s*\eta)*\xi\sim\s*(\eta*\xi)$ also holds.
\end{proof}

\bibliographystyle{ieeetr}
\footnotesize
\bibliography{SOn_Sn}

%
%

\end{document}